\documentclass[reqno]{amsart}
\usepackage{amssymb,mathrsfs}

\numberwithin{equation}{section}

\newtheorem{theorem}[equation]{Theorem}
\newtheorem{proposition}[equation]{Proposition}
\newtheorem{lemma}[equation]{Lemma}

\theoremstyle{definition}
\newtheorem{definition}[equation]{Definition}
\newtheorem{remark}[equation]{Remark}

\DeclareMathOperator{\Op}{Op}
\DeclareMathOperator{\supp}{supp}
\DeclareMathOperator{\Tr}{Tr}

\begin{document}
\title[Abstract edge pseudodifferential operators]{A calculus of abstract edge pseudodifferential operators of type $\varrho,\delta$}

\thanks{This material is based in part upon work supported by the National Science Foundation under Grant No. DMS-0901202.}

\author{Thomas Krainer}
\address{Penn State Altoona\\ 3000 Ivyside Park \\ Altoona, PA 16601-3760}
\email{krainer@psu.edu}

\begin{abstract}
In this paper we expand on B.-W.~Schulze's abstract edge pseudo\-differential calculus and introduce a larger class of operators that is modeled on H{\"o}rmander's $\varrho,\delta$ calculus, where $0 \leq \delta < \varrho \leq 1$. This expansion is motivated by recent work on boundary value problems for elliptic wedge operators with variable indicial roots by G.~Mendoza and the author, where operators of type $1,\delta$ for $0 < \delta < 1$ appear naturally.
Some of the results of this paper also represent improvements over the existing literature on the standard abstract edge calculus of operators of type $1,0$, such as trace class mapping properties of operators in abstract wedge Sobolev spaces. The presentation in this paper is largely self-contained to allow for an independent reading.
\end{abstract}


\maketitle


\section{Introduction}

In \cite{TKSchu89}, B.-W. Schulze introduced a calculus of pseudodifferential operators with operator valued symbols that satisfy symbol estimates that are twisted by strongly continuous group actions. This calculus, which is generally referred to as the abstract edge pseudodifferential calculus, turned out to be very useful to describe and analyze the local structure of elliptic partial differential equations and their parametrices on manifolds with boundary near the boundary, and, more generally, on manifolds with incomplete edge singularities near the edges, thus motivating the name for the calculus. The abstract edge calculus has been further developed over the years by many contributors, see for example \cite{TKDorschfeldtSchulze94,TKDorschfeldtGriemeSchulze97,TKHirschmann,TKSchroheBook,TKSchuNH,TKSchuWil,TKSeilerDiss,TKSeilerCont}.
A feature of the symbol estimates that is common to all these works is that differentiation with respect to the variables does not worsen the growth estimates (i.e. the order) in the covariables. In other words, the symbols exhibit estimates resembling H{\"o}rmander's type $\varrho,0$ symbols, twisted by strongly continuous group actions, where generally $\varrho = 1$ (except in J.~Seiler's boundedness theorem from \cite{TKSeilerCont} where type $0,0$ symbols are considered).

In this paper we expand on the abstract edge pseudodifferential calculus and introduce a larger class of operators that is modeled on H{\"o}rmander's $\varrho,\delta$ calculus, where $0 \leq \delta < \varrho \leq 1$. This expansion is motivated by recent work on boundary value problems for elliptic wedge operators with variable indicial roots  \cite{TKMeKrKernelBundle,TKMeKrVariableOrder,TKMeKrFirstOrder}, where operators of type $1,\delta$ for $0 < \delta < 1$ appear naturally. Some of the results obtained here for the general $\varrho,\delta$ class also represent improvements over the existing literature on the standard abstract edge calculus of operators of type $1,0$. These improvements pertain, in particular, to trace class mapping properties of operators in abstract wedge Sobolev spaces, and independence of the operator class with respect to any $\tau$-quantization (see \cite{TKShubin}) of the symbols, in particular with respect to the Kohn-Nirenberg ($\tau = 0$) and Weyl ($\tau = 1/2$) quantizations. The presentation in this paper is largely self-contained. Where appropriate, some key arguments of the existing literature are reproduced to allow for an independent reading.

\bigskip

\noindent
The structure of the paper is as follows:

In Section~\ref{TKsec-AbstractWSpaceinRq} we review the definition and some of the properties of the abstract wedge Sobolev space of $E$-valued distributions, where $E$ is a Hilbert space. Unless stated otherwise, all Hilbert spaces in this work are assumed to be complex and separable. We also give an alternative characterization of the wedge Sobolev spaces that is based on Littlewood-Paley theory. This characterization proves to be very useful to study mapping properties of pseudodifferential operators in such spaces. Effectively, this characterization and the composition theorem in the abstract edge calculus will allow us to reduce most arguments to ordinary pseudodifferential operators (i.e. those based on symbols that are not twisted by group actions) acting in Hilbert space valued $L^2$-spaces.

Section~\ref{TKsec-Pseudos} contains the elements of the abstract edge calculus of operators of type $\varrho,\delta$. We allow the symbols to exhibit polynomial growth in the variables and to take values in general Banach operator ideals. The latter is useful for applications in index and spectral theory, especially symbols taking values in Schatten-von Neumann classes. We consider general $\tau$-quantizations, see \cite{TKShubin}, and prove independence of the pseudodifferential operator class from the specific $\tau$-quantization. We stick to the range $0 \leq \delta < \varrho \leq 1$ because we want to retain asymptotic expansion formulas for the operations of the calculus.

Sections~\ref{TKsec-CompactOperators} and \ref{TKsec-TraceClass} are concerned with mapping properties of the pseudodifferential operators from Section~\ref{TKsec-Pseudos} in the scale of abstract wedge Sobolev spaces. Section~\ref{TKsec-CompactOperators} addresses boundedness and compactness, while Section~\ref{TKsec-TraceClass} is concerned with trace class mapping properties (i.e., membership in the Schatten-von Neumann class ${\mathscr C}_1$).

Finally, we include an appendix on Banach operator ideals in the Hilbert space category. In this appendix we also review some of the properties of the Schatten-von Neumann classes.


\section{Abstract wedge Sobolev spaces}\label{TKsec-AbstractWSpaceinRq}

Abstract wedge Sobolev spaces were introduced by Schulze in \cite{TKSchu89}, see also \cite{TKHirschmann,TKSchuNH,TKSchuWil}. In this section we review the standard definition of these function spaces, and give a different elementary characterization for them that is based on Littlewood-Paley theory. This characterization is useful since it allows to reduce the analysis of operators acting in abstract wedge Sobolev spaces to standard analysis of operators in $L^2$-spaces. We will take advantage of this in subsequent sections.

\medskip

Let $E$ be a Hilbert space, and let $\{\kappa_{\varrho}\}_{\varrho > 0}$ be a strongly continuous group action on $E$, i.e., ${\mathbb R}_+ \ni \varrho \mapsto \kappa_{\varrho} \in {\mathscr L}(E)$ is a representation of the multiplicative group $({\mathbb R}_+,\cdot)$ in ${\mathscr L}(E)$ that is continuous with respect to the strong operator topology. By the uniform boundedness principle, there exist constants $c,M \geq 0$ such that
\begin{equation}\label{TKkappagrowth}
\|\kappa_{\varrho}\|_{{\mathscr L}(E)} \leq c \max\{\varrho,\varrho^{-1}\}^M
\end{equation}
for all $\varrho > 0$.

\begin{definition}[{\cite[Section 3.1.2]{TKSchuNH}}]\label{TKWspacedef}
For $s \in {\mathbb R}$ the abstract wedge Sobolev space is defined as
\begin{gather*}
{\mathcal W}^s({\mathbb R}^q,(E,\kappa)) = \{u \in {\mathscr S}'({\mathbb R}^q,E) : {\mathscr F} u \textup{ is regular, and } \\
[\eta \mapsto \langle \eta \rangle^s\kappa_{\langle \eta \rangle}^{-1}{\mathscr F} u(\eta)] \in L^2({\mathbb R}^q,E)\}.
\end{gather*}
Here ${\mathscr F}$ denotes the Fourier transform on ${\mathscr S}'({\mathbb R}^q,E)$, and $\langle \eta \rangle = (1 + |\eta|^2)^{1/2}$, where $|\eta|$ is the Euclidean norm of $\eta \in {\mathbb R}^q$. It is custom to simply write ${\mathcal W}^s({\mathbb R}^q,E)$ if the group action $\kappa_{\varrho}$ on $E$ is clear from the context.
\end{definition}

${\mathcal W}^s({\mathbb R}^q,E)$ is a Hilbert space with inner product
$$
\langle u,v \rangle_{{\mathcal W}^s} = \int_{{\mathbb R}^q} \langle \eta \rangle^{2s}\langle \kappa^{-1}_{\langle\eta\rangle}{\mathscr F} u(\eta), \kappa^{-1}_{\langle\eta\rangle}{\mathscr F} v(\eta)\rangle_{E}\,d\eta.
$$
The space ${\mathscr S}({\mathbb R}^q,E)$ of rapidly decreasing $E$-valued functions is dense in the space ${\mathcal W}^s({\mathbb R}^q,E)$, and we have
$$
H^{s+M}({\mathbb R}^q,E) \subset {\mathcal W}^s({\mathbb R}^q,E) \subset H^{s-M}({\mathbb R}^q,E)
$$
with continuous embeddings with the growth constant $M \geq 0$ for the group action $\{\kappa_{\varrho}\}_{\varrho > 0}$ from \eqref{TKkappagrowth}. Note that multiplication by the strongly continuous operator functions $\eta \mapsto \kappa_{\langle \eta \rangle}$ and $\eta \mapsto \kappa^{-1}_{\langle \eta \rangle}$ preserves strong measurability of $E$-valued functions, and that both these operator functions are of tempered growth by \eqref{TKkappagrowth}. The density of ${\mathscr S}({\mathbb R}^q,E)$ in ${\mathcal W}^s({\mathbb R}^q,E)$ follows most easily from the characterization of the ${\mathcal W}^s$-spaces based on Littlewood-Paley theory that is given below.  Alternatively, it follows from the density of $C^{\infty}$-elements for the group action $\kappa_{\varrho}$ in $E$ and a tensor product argument.

The ${\mathcal W}^s$-spaces form a function space scale based on ${\mathcal W}^0({\mathbb R}^q,E)$ in the sense that
\begin{equation}\label{TKliftoperator}
\langle D_y \rangle^{\mu} = {\mathscr F}^{-1}_{\eta \to y} \langle \eta \rangle^{\mu} {\mathscr F}_{y'\to\eta} : {\mathcal W}^s({\mathbb R}^q,E) \to {\mathcal W}^{s-\mu}({\mathbb R}^q,E)
\end{equation}
is an isomorphism for all $s,\mu \in {\mathbb R}$. Moreover, setting $\kappa^{(s)}_{\varrho} = \varrho^{-s}\kappa_{\varrho}$, we also see that
\begin{equation}\label{TKWsequalsW0}
{\mathcal W}^s({\mathbb R}^q,(E,\kappa)) = {\mathcal W}^0({\mathbb R}^q,(E,\kappa^{(s)})).
\end{equation}
This allows to reduce many considerations regarding the ${\mathcal W}^s$-scale to the case when $s = 0$.

The following class of operator valued symbols is well adapted for Fourier multipliers in the ${\mathcal W}^s$-space scale.

\begin{definition}[{\cite[Section 3.2.1]{TKSchuNH}}]\label{TKMultipliersymbols}
Let $(E,\kappa_{\varrho})$ and $(\tilde{E},\tilde{\kappa}_{\varrho})$ be Hilbert spaces that come equipped with strongly continuous group actions. For $\mu \in {\mathbb R}$ define
$$
S^{\mu}_{1,0}({\mathbb R}^q;(E,\kappa),(\tilde{E},\tilde{\kappa}))
$$
as the space of all $a(\eta) \in C^{\infty}({\mathbb R}^q,{\mathscr L}(E,\tilde{E}))$ such that for all $\alpha \in {\mathbb N}_0^q$ there exists a constant $C_{\alpha} > 0$ such that
$$
\|\tilde{\kappa}^{-1}_{\langle \eta \rangle}[\partial_{\eta}^{\alpha}a(\eta)]\kappa_{\langle \eta \rangle}\|_{{\mathscr L}(E,\tilde{E})} \leq C_{\alpha}\langle \eta \rangle^{\mu-|\alpha|}
$$
for all $\eta \in {\mathbb R}^q$.

If any of the Hilbert spaces $E$ or $\tilde{E}$ carries the trivial group action $\kappa_{\varrho} \equiv \textup{Id}_E$ or $\tilde{\kappa}_{\varrho} \equiv \textup{Id}_{\tilde{E}}$, then that action is dropped from the notation of the symbol class, i.e., one simply writes $S^{\mu}_{1,0}({\mathbb R}^q;E,(\tilde{E},\tilde{\kappa}))$ or $S^{\mu}_{1,0}({\mathbb R}^q;(E,\kappa),\tilde{E})$ (or both).
\end{definition}

For any $a \in S^{\mu}_{1,0}({\mathbb R}^q;(E,\kappa),(\tilde{E},\tilde{\kappa}))$ the Fourier multiplier
\begin{equation}\label{TKFourierMultiplierWsSpace}
{\mathscr F}^{-1}_{\eta \to y} a(\eta) {\mathscr F}_{y'\to\eta} : {\mathcal W}^s({\mathbb R}^q,E) \to {\mathcal W}^{s-\mu}({\mathbb R}^q,\tilde{E})
\end{equation}
is evidently continuous (the estimates on the derivatives on $a(\eta)$ from Definition~\ref{TKMultipliersymbols} are not required for this). The operator \eqref{TKliftoperator} is a special case of \eqref{TKFourierMultiplierWsSpace} with $a(\eta) = \langle \eta \rangle^{\mu}\textup{Id}_E \in S^{\mu}_{1,0}({\mathbb R}^q;(E,\kappa),(E,\kappa))$.

\medskip

We now proceed to give an elementary characterization of the ${\mathcal W}^0$-space that is based on Littlewood-Paley theory. As already mentioned, this characterization is useful since it allows to reduce Fourier analytic considerations of operators acting in the ${\mathcal W}^s$-function space scale to standard Fourier analysis of operators in $L^2$-spaces.

Let $\phi_0 \in C^{\infty}({\mathbb R}^q)$ such that $\phi_0 \equiv 1$ in a neighborhood of $|\eta| \leq 1$, $\supp(\phi_0) \subset \{\eta : |\eta| < 2\}$, and $0 \leq \phi_0 \leq 1$. Define
$$
\varphi_j(\eta) = \phi_0(2^{-j}\eta) - \phi_0(2^{-j+1}\eta)
$$
for $j \in {\mathbb N}$. Then the $\{\phi_j\}_{j=0}^{\infty}$ form a dyadic resolution of the identity on ${\mathbb R}^q$ with the following properties:
\begin{enumerate}
\item $\phi_j \in C^{\infty}({\mathbb R}^q)$ with $\supp(\phi_j) \subset \{\eta : 2^{j-1} < |\eta| < 2^{j+1}\}$ for $j \in {\mathbb N}$.
\item $0 \leq \phi_j \leq 1$.
\item $\sum_{j=0}^{\infty}\phi_j(\eta) = 1$ for all $\eta \in {\mathbb R}^q$. This is a locally finite sum, and for each $\eta \in {\mathbb R}^q$ at most two consecutive summands are nonzero.
\item For every multi-index $\alpha \in {\mathbb N}_0^q$ there exists a constant $C_{\alpha} > 0$ that is independent of $j \in {\mathbb N}_0$ such that $|\partial^{\alpha}_{\eta}\phi_j(\eta)| \leq C_{\alpha}2^{-j|\alpha|}$ for all $\eta \in {\mathbb R}^q$.

Because $|\eta|/2^j \leq 2$ on $\supp(\phi_j)$ this implies the estimate 
$$
|\partial_{\eta}^{\alpha}\phi_j(\eta)| \leq 2^{|\alpha|}C_{\alpha}|\eta|^{-|\alpha|}
$$
for $|\eta| \geq 1$.
\end{enumerate}

Let $\psi_0,\psi_1 \in C^{\infty}({\mathbb R}^q)$ such that $\supp(\psi_0) \subset \{\eta : |\eta| < 2\}$ and $\psi_0 \equiv 1$ in a neighborhood of $\supp(\phi_0)$, and correspondingly $\supp(\psi_1) \subset \{\eta : 1 < |\eta| < 4\}$ with $\psi_1 \equiv 1$ in a neighborhood of $\supp(\phi_1)$. Define $\psi_j(\eta) = \psi_1(2^{-j+1}\eta)$ for $j = 2,3,\ldots$ Then $\supp(\psi_j) \subset \{\eta : 2^{j-1} < |\eta| < 2^{j+1}\}$, and $\psi_j \equiv 1$ in a neighborhood of $\supp(\phi_j)$ for all $j \in {\mathbb N}$. Moreover, for every multi-index $\alpha \in {\mathbb N}_0^q$ there exists a constant $D_{\alpha} > 0$ that is independent of $j \in {\mathbb N}_0$ such that $|\partial^{\alpha}_{\eta}\psi_j(\eta)| \leq D_{\alpha}2^{-j|\alpha|}$ for all $\eta \in {\mathbb R}^q$, and consequently $|\partial^{\alpha}_{\eta}\psi_j(\eta)| \leq 2^{|\alpha|} D_{\alpha}|\eta|^{-|\alpha|}$ for $|\eta| \geq 1$.

Let $\ell^2({\mathbb N}_0,E)$ be the Hilbert space of square summable sequences with entries in $E$ with inner product
$$
\langle (e_k)_{k=0}^{\infty}, (f_k)_{k=0}^{\infty} \rangle_{\ell^2} = \sum_{k=0}^{\infty} \langle e_k,f_k \rangle_E.
$$
For $j \in {\mathbb N}_0$ define
$$
\iota_j : E \to \ell^2({\mathbb N}_0,E), \; \iota_j(e) = (0,\ldots,0,e,0,\ldots),
$$
where the element $e \in E$ is injected as the $j$-th component of the sequence, and
$$
\pi_j : \ell^2({\mathbb N}_0,E) \to E, \; \pi_j\bigl[(e_k)_{k=0}^{\infty}\bigr] = e_j.
$$
Then $\iota_j \in {\mathscr L}(E,\ell^2({\mathbb N}_0,E))$ and $\pi_j \in {\mathscr L}(\ell^2({\mathbb N}_0,E),E)$ with operator norm $\|\iota_j\| = \|\pi_j\| = 1$ for all $j$, and $\pi_k\iota_j = \delta_{jk}\textup{Id}_E$ with the Kronecker delta $\delta_{jk}$.

\begin{lemma}\label{TKSTSymbols}
For $\eta \in {\mathbb R}^q$ let
\begin{align*}
s(\eta) &= \sum_{j=0}^{\infty}\phi_j(\eta)\iota_j\kappa_{2^{-j}} : E \to \ell^2({\mathbb N}_0,E), \\
t(\eta) &= \sum_{k=0}^{\infty}\psi_k(\eta)\kappa_{2^k}\pi_k : \ell^2({\mathbb N}_0,E) \to E.
\end{align*}
Then
\begin{align*}
s &\in S_{1,0}^0({\mathbb R}^q,(E,\kappa),\ell^2({\mathbb N}_0,E)), \\
t &\in S_{1,0}^0({\mathbb R}^q,\ell^2({\mathbb N}_0,E),(E,\kappa)),
\end{align*}
and $t(\eta)s(\eta) = \textup{Id}_E$ for all $\eta \in {\mathbb R}^q$.
\end{lemma}
\begin{proof}
Since both sums are locally finite, $s(\eta)$ and $t(\eta)$ are $C^{\infty}$ operator functions. We have
$$
t(\eta)s(\eta) = \sum_{k,j=0}^{\infty}\psi_k(\eta)\phi_j(\eta)\underbrace{\kappa_{2^k}\pi_k \iota_j\kappa_{2^{-j}}}_{=\delta_{jk}\textup{Id}_E} = \sum_{j=0}^{\infty}\underbrace{\psi_j(\eta)\phi_j(\eta)}_{=\phi_j(\eta)}\textup{Id}_E = \textup{Id}_E.
$$
It remains to show the symbol estimates for $s(\eta)$ and $t(\eta)$. For $|\eta| \geq 2$ we have
\begin{align*}
\|[\partial_{\eta}^{\alpha}s(\eta)]\kappa_{|\eta|}\| &= \Bigl\|\sum_{j=0}^{\infty}\partial_{\eta}^{\alpha}\phi_j(\eta)\iota_j\kappa_{|2^{-j}\eta|}\Bigr\| \\
&\leq \sum_{j=0}^{\infty}|\partial_{\eta}^{\alpha}\phi_j(\eta)|\|\kappa_{|2^{-j}\eta|}\| \leq 2^{|\alpha|+1}C_{\alpha}K |\eta|^{-|\alpha|}, \\
\|\kappa^{-1}_{|\eta|}[\partial_{\eta}^{\alpha}t(\eta)]\| &= \Bigl\|\sum_{k=0}^{\infty}\partial_{\eta}^{\alpha}\psi_k(\eta)\kappa_{|2^{-k}\eta|^{-1}}\pi_k \Bigr\| \\
&\leq \sum_{k=0}^{\infty}|\partial_{\eta}^{\alpha}\psi_k(\eta)| \|\kappa_{|2^{-k}\eta|^{-1}} \| \leq 2^{|\alpha|+1}D_{\alpha}K|\eta|^{-|\alpha|},
\end{align*}
where $K = \sup\{\|\kappa_{\varrho}\| : 2^{-1} \leq \varrho \leq 2\} < \infty$ (recall that this is finite by the uniform boundedness principle in view of the strong continuity of $\kappa_{\varrho}$). Note that $|2^{-j}\eta|$ and $|2^{-k}\eta|^{-1}$ are in the interval $[2^{-1},2]$ for all $\eta$ in the support of $\partial^{\alpha}_{\eta}\phi_j$ and $\partial^{\alpha}_{\eta}\psi_k$, respectively, and that at most two summands in the infinite sums are nonzero for each $\eta$ by construction of the $\phi_j$ and $\psi_k$.
\end{proof}

By Lemma~\ref{TKSTSymbols} and \eqref{TKFourierMultiplierWsSpace},
\begin{equation}\label{TKSTOperators}
\begin{aligned}
S &= \Op(s) = {\mathscr F}^{-1}_{\eta\to y}s(\eta){\mathscr F}_{y'\to\eta} : {\mathcal W}^0({\mathbb R}^q,E) \to L^2({\mathbb R}^q,\ell^2({\mathbb N}_0,E)) \\
T &= \Op(t) = {\mathscr F}^{-1}_{\eta\to y}t(\eta){\mathscr F}_{y'\to\eta} : L^2({\mathbb R}^q,\ell^2({\mathbb N}_0,E)) \to {\mathcal W}^0({\mathbb R}^q,E)
\end{aligned}
\end{equation}
are continuous, and $T\circ S = \textup{Id}_{{\mathcal W}^0}$. This shows that $S$ maps ${\mathcal W}^0({\mathbb R}^q,E)$ isomorphically to a closed subspace of $L^2({\mathbb R}^q,\ell^2({\mathbb N}_0,E))$, and thus $\|Su\|_{L^2({\mathbb R}^q,\ell^2({\mathbb N}_0,E))}$ for $u \in {\mathcal W}^0({\mathbb R}^q,E)$ is an equivalent norm on ${\mathcal W}^0({\mathbb R}^q,E)$. Consequently, we have proved that ${\mathcal W}^0({\mathbb R}^q,E)$ is isomorphic to a generalized $E$-valued Triebel-Lizorkin space (or $E$-valued Besov space which is equivalent here) with operator-valued weight given by the inverse group action $\kappa^{-1}_{\varrho}$. More generally, it would make sense to define spaces $B^s_{p,q}({\mathbb R}^q,(E,\kappa))$ and $F^s_{p,q}({\mathbb R}^q,(E,\kappa))$ with that operator-valued weight, generalizing the standard $E$-valued Besov and Triebel-Lizorkin spaces (see \cite{TKAmann97,TKSSS12,TKSchmeisser87,TKTriebel97} for the latter). In this case we would recover
$$
{\mathcal W}^s({\mathbb R}^q,(E,\kappa)) = B^s_{2,2}({\mathbb R}^q,(E,\kappa)) = F^s_{2,2}({\mathbb R}^q,(E,\kappa))
$$
from the above. However, we will not pursue this here, but merely summarize what we have proved as follows.

\begin{proposition}
A distribution $u \in {\mathscr S}'({\mathbb R}^q,E)$ belongs to ${\mathcal W}^0({\mathbb R}^q,E)$ if and only if $Su \in L^2({\mathbb R}^q,\ell^2({\mathbb N}_0,E))$ for the operator $S$ from \eqref{TKSTOperators}. More generally, $u \in {\mathcal W}^s({\mathbb R}^q,E)$ if and only if
$$
\sum_{j=0}^{\infty}\|2^{js}\kappa^{-1}_{2^j}\phi_j(D_y)u\|_{L^2({\mathbb R}^q,E)}^2 < \infty.
$$
\end{proposition}

In particular, as mentioned after Definition~\ref{TKWspacedef}, it is easy now to see the density of ${\mathscr S}({\mathbb R}^q,E)$ in ${\mathcal W}^s({\mathbb R}^q,E)$: In view of \eqref{TKWsequalsW0} we just have to prove this for $s = 0$. Note that the Fourier multipliers $S$ and $T$ from \eqref{TKSTOperators} map rapidly decreasing functions to rapidly decreasing functions. Thus, for $u \in {\mathcal W}^0({\mathbb R}^q,E)$, choose a sequence $(v_k)_k$ in ${\mathscr S}({\mathbb R}^q,\ell^2({\mathbb N}_0,E))$ such that $v_k \to Su$ in $L^2({\mathbb R}^q,\ell^2({\mathbb N}_0,E))$. Then $u_k = Tv_k \in {\mathscr S}({\mathbb R}^q,E)$, and $u_k \to TSu = u$ in ${\mathcal W}^0({\mathbb R}^q,E)$.

\medskip

The operators $S$ and $T$ from \eqref{TKSTOperators} are particularly useful to study functional analytic properties of operators in the ${\mathcal W}^s$-scale. To this end recall that ${\mathscr I} \subset {\mathscr L}$ is an operator ideal in the bounded operators acting between Hilbert spaces if for all Hilbert spaces $E$, $\tilde{E}$, and $\hat{E}$ we have that ${\mathscr I}(E,\tilde{E}) \subset {\mathscr L}(E,\tilde{E})$ is a subspace that contains the finite-rank operators, and for any $A \in {\mathscr I}(E,\tilde{E})$ and $G \in {\mathscr L}(\tilde{E},\hat{E})$ and $G' \in {\mathscr L}(\hat{E},E)$ we have $GA \in {\mathscr I}(E,\hat{E})$ and $AG' \in {\mathscr I}(\hat{E},\tilde{E})$. Examples of interest include ${\mathscr I} = {\mathscr L}$ (all bounded operators), ${\mathscr I} = {\mathscr K}$ (compact operators), and ${\mathscr I} = {\mathscr C}_p$ for $1 \leq p < \infty$ (the Banach operator ideal of Schatten--von Neumann operators with $p$-summable approximation numbers). See the appendix for additional information.

\begin{proposition}\label{TKBddnessWsScale}
Let $(E,\kappa)$ and $(\tilde{E},\tilde{\kappa})$ be Hilbert spaces with strongly continuous group actions $\{\kappa_{\varrho}\}_{\varrho > 0}$ and $\{\tilde{\kappa}_{\varrho}\}_{\varrho > 0}$, and let $S,T$ be the operators associated with $(E,\kappa)$ and $\tilde{S},\tilde{T}$ be the operators associated with $(\tilde{E},\tilde{\kappa})$ according to \eqref{TKSTOperators}, respectively.

Then an operator $A : {\mathscr S}({\mathbb R}^q,E) \to {\mathscr S}'({\mathbb R}^q,\tilde{E})$ is continuous in the spaces
$$
A : {\mathcal W}^0({\mathbb R}^q,E) \to {\mathcal W}^0({\mathbb R}^q,\tilde{E})
$$
if and only if
$$
\tilde{S} A T : L^2({\mathbb R}^q,\ell^2({\mathbb N}_0,E)) \to L^2({\mathbb R}^q,\ell^2({\mathbb N}_0,\tilde{E}))
$$
is continuous.

More generally, for any operator ideal in ${\mathscr I} \subset {\mathscr L}$ we have the equivalence
\begin{gather*}
A \in {\mathscr I}({\mathcal W}^0({\mathbb R}^q,E),{\mathcal W}^0({\mathbb R}^q,\tilde{E})) \\
\Longleftrightarrow \\
\tilde{S}A T \in {\mathscr I}(L^2({\mathbb R}^q,\ell^2({\mathbb N}_0,E)),L^2({\mathbb R}^q,\ell^2({\mathbb N}_0,\tilde{E}))).
\end{gather*}
\end{proposition}
\begin{proof}
If $A : {\mathcal W}^0({\mathbb R}^q,E) \to {\mathcal W}^0({\mathbb R}^q,\tilde{E})$ is continuous (or belongs to ${\mathscr I}$), then
\begin{equation}\label{TKtSAT}
\tilde{S}AT : L^2({\mathbb R}^q,\ell^2({\mathbb N}_0,E)) \to L^2({\mathbb R}^q,\ell^2({\mathbb N}_0,\tilde{E}))
\end{equation}
is continuous as a composition of continuous operators (or belongs to ${\mathscr I}$ because of the ideal property). Conversely, if \eqref{TKtSAT} is continuous (or belongs to ${\mathscr I}$), then
$$
A = \tilde{T}[\tilde{S}AT]S  : {\mathcal W}^0({\mathbb R}^q,E) \to {\mathcal W}^0({\mathbb R}^q,\tilde{E})
$$
is continuous as a composition of continuous operators (or belongs to ${\mathscr I}$ because of the ideal property).
\end{proof}

We conclude this section with the definition of weighted abstract wedge Sobolev spaces spaces, see \cite{TKDorschfeldtGriemeSchulze97,TKSeilerDiss}.

\begin{definition}
For $\vec{s} = (s_1,s_2) \in {\mathbb R}^2$ define
$$
{\mathcal W}^{\vec{s}}({\mathbb R}^q,E) = \langle y \rangle^{-s_2}{\mathcal W}^{s_1}({\mathbb R}^q,E).
$$
\end{definition}


\section{Abstract edge pseudodifferential operators}\label{TKsec-Pseudos}

Operator valued symbols twisted by strongly continuous group actions and associated pseudodifferential operators of type $1,0$ were introduced by Schulze \cite{TKSchu89}. The calculus has been further developed and applied in \cite{TKDorschfeldtSchulze94,TKDorschfeldtGriemeSchulze97,TKHirschmann,TKSchroheBook,TKSchuNH,TKSchuWil,TKSeilerDiss,TKSeilerCont}. Recent work \cite{TKMeKrKernelBundle,TKMeKrVariableOrder,TKMeKrFirstOrder} on boundary value problems for elliptic wedge differential operators with variable indicial roots provides a motivation for considering symbols and operators of more general types, in particular those of type $1,\delta$ with $\delta > 0$. In this section we therefore introduce a class of global operator valued symbols and associated operators of type $\varrho,\delta$, $0 \leq \delta < \varrho \leq 1$, and discuss the elements of the calculus for this class. We allow the symbols to take values in general Banach operator ideals ${\mathscr I}$ in the Hilbert space category, and prove independence of the pseudodifferential operator class with respect to all $\tau$-quantizations \cite{TKShubin}, thus acquiring more flexibility for applications of the calculus.

We will utilize Kumano-go's technique \cite{TKKumanogoBook}, similar to the works \cite{TKDorschfeldtGriemeSchulze97,TKSchroheBook,TKSeilerDiss,TKSeilerCont} that discuss various classes of twisted pseudodifferential operators of type $1,0$.

\begin{definition}[Twisted operator valued symbols]
Fix $0 \leq \delta < \varrho \leq 1$.

For $\vec{\mu} = (\mu_1,\mu_2) \in {\mathbb R}^2$ let $S^{\vec{\mu}}_{\varrho,\delta}({\mathbb R}^q\times{\mathbb R}^q;(E,\kappa),(\tilde{E},\tilde{\kappa}))_{{\mathscr I}}$ be the space of all $a(y,\eta) \in C^{\infty}({\mathbb R}^q\times{\mathbb R}^q,{\mathscr I}(E,\tilde{E}))$ such that for each $\alpha,\beta \in {\mathbb N}_0^q$ there exists a constant $C_{\alpha,\beta} \geq 0$ with
$$
\|\tilde{\kappa}_{\langle \eta \rangle}^{-1}[D_y^{\beta}\partial_{\eta}^{\alpha}a(y,\eta)]\kappa_{\langle \eta \rangle}\|_{{\mathscr I}(E,\tilde{E})} \leq C_{\alpha,\beta} \langle y \rangle^{\mu_2}\langle \eta \rangle^{\mu_1-\varrho|\alpha| + \delta|\beta|}
$$
for all $(y,\eta) \in {\mathbb R}^q\times{\mathbb R}^q$.

Moreover, for $\vec{\mu} = (\mu_1,\mu_2,\mu_3) \in {\mathbb R}^3$ let $S^{\vec{\mu}}_{\varrho,\delta}({\mathbb R}^q\times{\mathbb R}^q\times{\mathbb R}^q;(E,\kappa),(\tilde{E},\tilde{\kappa}))_{{\mathscr I}}$ be the space of all $a(y,y',\eta) \in C^{\infty}({\mathbb R}^q\times{\mathbb R}^q\times{\mathbb R}^q,{\mathscr I}(E,\tilde{E}))$ such that for each $\alpha,\beta,\gamma \in {\mathbb N}_0^q$ there exists a constant $C_{\alpha,\beta,\gamma} \geq 0$ with
$$
\|\tilde{\kappa}_{\langle \eta \rangle}^{-1}[D_y^{\beta}D_{y'}^{\gamma}\partial_{\eta}^{\alpha}a(y,y',\eta)]\kappa_{\langle \eta \rangle}\|_{{\mathscr I}(E,\tilde{E})} \leq C_{\alpha,\beta,\gamma} \langle y \rangle^{\mu_2}\langle y' \rangle^{\mu_3}\langle \eta \rangle^{\mu_1-\varrho|\alpha| + \delta(|\beta|+|\gamma|)}
$$
for all $(y,y,'\eta) \in {\mathbb R}^q\times{\mathbb R}^q\times{\mathbb R}^q$.
\end{definition}

The symbol spaces carry natural Fr{\'e}chet topologies induced by the seminorms given by the best constants in the estimates.

In the case ${\mathscr I} = {\mathscr L}$ the operator ideal subscript is dropped from the symbol space notation. Likewise, if $E$ or $\tilde{E}$ carry the trivial group action $\kappa_{\varrho} = \textup{Id}_E$ or $\tilde{\kappa}_{\varrho} = \textup{Id}_{\tilde{E}}$, respectively, we will just write $E$ instead of the pair $(E,\textup{Id}_E)$ or $\tilde{E}$ instead of $(\tilde{E},\textup{Id}_{\tilde{E}})$.
If $\mu_2 = 0$ (for simple symbols) or $\mu_2=\mu_3=0$ (for double symbols) it is customary to revert to the single superscript $\mu_1$ instead of the vector $(\mu_1,0)$ or $(\mu_1,0,0)$ to indicate the order of the symbol.

A technically convenient observation for the twisted calculus is that there are trivial inclusions
\begin{align*}
S^{\vec{\mu}-(M+\tilde{M},0,0)}_{\varrho,\delta}({\mathbb R}^q\times{\mathbb R}^q\times{\mathbb R}^q;E,\tilde{E})_{{\mathscr I}} &\subset
S^{\vec{\mu}}_{\varrho,\delta}({\mathbb R}^q\times{\mathbb R}^q\times{\mathbb R}^q;(E,\kappa),(\tilde{E},\tilde{\kappa}))_{{\mathscr I}} \\
&\subset S^{\vec{\mu}+(M+\tilde{M},0,0)}_{\varrho,\delta}({\mathbb R}^q\times{\mathbb R}^q\times{\mathbb R}^q;E,\tilde{E})_{{\mathscr I}}, \\
\intertext{as well as}
S^{\vec{\mu}}_{\varrho,\delta}({\mathbb R}^q\times{\mathbb R}^q\times{\mathbb R}^q;E,\tilde{E})_{{\mathscr I}} &\subset
S^{\vec{\mu}}_{\varrho,\delta}({\mathbb R}^q\times{\mathbb R}^q\times{\mathbb R}^q;E,\tilde{E}),
\end{align*}
where $M$ and $\tilde{M}$ are the growth constants for $\{\kappa_{\varrho}\}_{\varrho > 0}$ and $\{\tilde{\kappa}_{\varrho}\}_{\varrho > 0}$ in \eqref{TKkappagrowth}, respectively. Hence the impact of the group actions is in essence limited to determining a filtration by order that is different from the standard filtration of the H{\"o}rmander calculus of type $\varrho,\delta$ \cite{TKHormanderVol3,TKKumanogoBook}.

\bigskip

With any double symbol $a(y,y',\eta) \in S^{\vec{\mu}}_{\varrho,\delta}({\mathbb R}^q\times{\mathbb R}^q\times{\mathbb R}^q;(E,\kappa),(\tilde{E},\tilde{\kappa}))$ we associate a pseudodifferential operator $\Op(a) : {\mathscr S}({\mathbb R}^q,E) \to {\mathscr S}({\mathbb R}^q,\tilde{E})$ as usual via
$$
[\Op(a)u](y) = (2\pi)^{-q}\int_{{\mathbb R}^q}\int_{{\mathbb R}^q} e^{i(y-y')\eta}a(y,y',\eta)u(y')\,dy'\,d\eta.
$$
Moreover, for each $\tau \in {\mathbb R}$ and $a \in S^{\vec{\mu}}_{\varrho,\delta}({\mathbb R}^q\times{\mathbb R}^q;(E,\kappa),(\tilde{E},\tilde{\kappa}))$, we associate a $\tau$-quantized pseudodifferential operator $\Op_{\tau}(a) : {\mathscr S}({\mathbb R}^q,E) \to {\mathscr S}({\mathbb R}^q,\tilde{E})$ via
$$
\Op_{\tau}(a)u(y) = (2\pi)^{-q}\int_{{\mathbb R}^q}\int_{{\mathbb R}^q}e^{i(y-y')\eta}a(\tau y' + (1-\tau)y,\eta)u(y')\,dy'\,d\eta.
$$
Note that $\Op_0(a) = a(Y,D)$ is the Kohn-Nirenberg quantized pseudodifferential operator with (left) symbol $a(y,\eta)$, for $\tau = 1/2$ we obtain a Weyl quantized operator $\Op_{1/2}(a)=a^{w}(Y,D)$ with Weyl symbol $a$, and for $\tau = 1$ we obtain the quantization of the right symbol $a(y',\eta)$.

For $\vec{\mu} = (\mu_1,\mu_2)$ define
\begin{align*}
\Psi^{\vec{\mu}}_{\varrho,\delta}({\mathbb R}^q;(E,\kappa),(\tilde{E},\tilde{\kappa}))_{{\mathscr I}} = \{\Op_{\tau}(a) &: {\mathscr S}({\mathbb R}^q,E) \to {\mathscr S}({\mathbb R}^q,\tilde{E}) : \\
&a \in S^{\vec{\mu}}_{\varrho,\delta}({\mathbb R}^q\times{\mathbb R}^q;(E,\kappa),(\tilde{E},\tilde{\kappa}))_{{\mathscr I}}\}.
\end{align*}
We will see in Theorem~\ref{TKOperatorClassIndependentofTau} below that $\Psi^{\vec{\mu}}_{\varrho,\delta}({\mathbb R}^q;(E,\kappa),(\tilde{E},\tilde{\kappa}))_{{\mathscr I}}$ is independent of $\tau \in {\mathbb R}$, and that the $\tau$-quantization map
$$
S^{\vec{\mu}}_{\varrho,\delta}({\mathbb R}^q\times{\mathbb R}^q;(E,\kappa),(\tilde{E},\tilde{\kappa}))_{{\mathscr I}} \ni a \mapsto \Op_{\tau}(a) \in \Psi^{\vec{\mu}}_{\varrho,\delta}({\mathbb R}^q;(E,\kappa),(\tilde{E},\tilde{\kappa}))_{{\mathscr I}}
$$
is a bijection. We equip the operator space $\Psi^{\vec{\mu}}_{\varrho,\delta}({\mathbb R}^q;(E,\kappa),(\tilde{E},\tilde{\kappa}))_{{\mathscr I}}$ with the symbol topology transferred by any $\tau$-quantization map, making it a Fr{\'e}chet space. This topology turns out to be independent of $\tau$.

\begin{lemma}\label{TKTechnicalLemma}
Fix $\tau \in {\mathbb R}$. Let $a(y,y',\eta) \in S^{(\mu_1,\mu_2,\mu_3)}_{\varrho,\delta}({\mathbb R}^q\times{\mathbb R}^q\times{\mathbb R}^q;(E,\kappa),(\tilde{E},\tilde{\kappa}))_{{\mathscr I}}$, and define
$$
a^{\tau}_{\theta}(y,\eta) = \frac{1}{(2\pi)^q}\iint e^{-ix\xi}a(y-\tau x,y+(1-\tau)x,\eta+\theta\xi)\,dx\,d\xi, \; 0 \leq \theta \leq 1,
$$
where the integral is an ${\mathscr I}(E,\tilde{E})$-valued oscillatory integral.

Then $a^{\tau}_{\theta} \in C([0,1]_{\theta},S^{(\mu_1,\mu_2+\mu_3)}_{\varrho,\delta}({\mathbb R}^q\times{\mathbb R}^q;(E,\kappa),(\tilde{E},\tilde{\kappa}))_{{\mathscr I}})$, and the map $a \mapsto a^{\tau}_{\theta}$ is continuous. For each $N \in {\mathbb N}$ we have
\begin{equation}\label{TKatauthetaexpansion}
a^{\tau}_{\theta}(y,\eta) = \sum_{|\alpha + \beta| < N}\frac{\theta^{|\alpha + \beta|}}{\alpha !\beta !}\tau^{|\alpha|}(1-\tau)^{|\beta|}\partial_{\eta}^{\alpha+\beta}(-D_y)^{\alpha}D_{y'}^{\beta}a(y,y',\eta)\big|_{y'=y} + r^{\tau}_{N,\theta}(y,\eta),
\end{equation}
where $r^{\tau}_{N,\theta}(y,\eta) \in C([0,1]_{\theta},S^{(\mu_1-(\varrho-\delta)N,\mu_2+\mu_3)}_{\varrho,\delta}({\mathbb R}^q\times{\mathbb R}^q;(E,\kappa),(\tilde{E},\tilde{\kappa}))_{{\mathscr I}})$ is given by
\begin{gather*}
r^{\tau}_{N,\theta}(y,\eta) = N\theta^N\sum_{|\alpha+\beta|=N}\int_0^1\frac{(1-s)^{N-1}}{\alpha !\beta !}\tau^{|\alpha|}(1-\tau)^{|\beta|} \; \cdot \\
\cdot \Bigl[\frac{1}{(2\pi)^q}\iint e^{-ix\xi}(\partial^{\alpha+\beta}_{\eta}(-D_y)^{\alpha}D_{y'}^{\beta}a)(y-\tau x,y+(1-\tau)x,\eta+s\theta\xi)\,dx\,d\xi\Bigr]\,ds.
\end{gather*}
\end{lemma}
\begin{proof}
We first argue that it is enough to prove the lemma in the case of trivial group actions $\kappa_{\varrho} \equiv \textup{Id}_E$ and $\tilde{\kappa}_{\varrho} \equiv \textup{Id}_{\tilde{E}}$. Thus suppose that the lemma is proved in that case. Let $M$ and $\tilde{M}$ be the growth constants for $\{\kappa_{\varrho}\}_{\varrho > 0}$ and $\{\tilde{\kappa}_{\varrho}\}_{\varrho > 0}$ in \eqref{TKkappagrowth}, respectively. Because
$$
S^{(\mu_1,\mu_2,\mu_3)}_{\varrho,\delta}({\mathbb R}^q\times{\mathbb R}^q\times{\mathbb R}^q;(E,\kappa),(\tilde{E},\tilde{\kappa}))_{{\mathscr I}} \subset S^{(\mu_1+M+\tilde{M},\mu_2,\mu_3)}_{\varrho,\delta}({\mathbb R}^q\times{\mathbb R}^q\times{\mathbb R}^q;E,\tilde{E})_{{\mathscr I}}
$$
we can apply the lemma to $a(y,y',\eta)$ as an element of the latter class. Choose $N \in {\mathbb N}$ so large that $(\varrho-\delta)N \geq 2(M + \tilde{M})$ and apply \eqref{TKatauthetaexpansion}. Then
$$
S^{(\mu_1+M+\tilde{M}-(\varrho-\delta)N,\mu_2+\mu_3)}_{\varrho,\delta}({\mathbb R}^q\times{\mathbb R}^q;E,\tilde{E})_{{\mathscr I}} \hookrightarrow S^{(\mu_1,\mu_2+\mu_3)}_{\varrho,\delta}({\mathbb R}^q\times{\mathbb R}^q;(E,\kappa),(\tilde{E},\tilde{\kappa}))_{{\mathscr I}},
$$
and consequently
$$
r^{\tau}_{N,\theta} \in C([0,1]_{\theta},S^{(\mu_1,\mu_2+\mu_3)}_{\varrho,\delta}({\mathbb R}^q\times{\mathbb R}^q;(E,\kappa),(\tilde{E},\tilde{\kappa}))_{{\mathscr I}}).
$$
But since
\begin{gather*}
\sum_{|\alpha + \beta| < N}\frac{\theta^{|\alpha + \beta|}}{\alpha !\beta !}\tau^{|\alpha|}(1-\tau)^{|\beta|}\partial_{\eta}^{\alpha+\beta}(-D_y)^{\alpha}D_{y'}^{\beta}a(y,y',\eta)\big|_{y'=y} \\
\in C([0,1]_{\theta},S^{(\mu_1,\mu_2+\mu_3)}_{\varrho,\delta}({\mathbb R}^q\times{\mathbb R}^q;(E,\kappa),(\tilde{E},\tilde{\kappa}))_{{\mathscr I}})
\end{gather*}
we also get that $a^{\tau}_{\theta} \in C([0,1]_{\theta},S^{(\mu_1,\mu_2+\mu_3)}_{\varrho,\delta}({\mathbb R}^q\times{\mathbb R}^q;(E,\kappa),(\tilde{E},\tilde{\kappa}))_{{\mathscr I}})$ as desired. Continuity of $a \mapsto a^{\tau}_{\theta}$ in case of general group actions follows from the case of trivial group actions via an application of the closed graph theorem. By what we have already shown we now obtain that, for all $\alpha,\beta \in {\mathbb N}_0^q$,
$$
(s,\theta) \mapsto \frac{1}{(2\pi)^q}\iint e^{-ix\xi}(\partial^{\alpha+\beta}_{\eta}(-D_y)^{\alpha}D_{y'}^{\beta}a)(y-\tau x,y+(1-\tau)x,\eta+s\theta\xi)\,dx\,d\xi
$$
is continuous on $[0,1]\times[0,1]$ taking values in
$$
S^{(\mu_1-(\varrho-\delta)|\alpha+\beta|,\mu_2+\mu_3)}_{\varrho,\delta}({\mathbb R}^q\times{\mathbb R}^q;(E,\kappa),(\tilde{E},\tilde{\kappa}))_{{\mathscr I}}.
$$
Thus the formula for the remainders $r^{\tau}_{N,\theta}(y,\eta)$ implies that
$$
r^{\tau}_{N,\theta}(y,\eta) \in C([0,1]_{\theta},S^{(\mu_1-(\varrho-\delta)N,\mu_2+\mu_3)}_{\varrho,\delta}({\mathbb R}^q\times{\mathbb R}^q;(E,\kappa),(\tilde{E},\tilde{\kappa}))_{{\mathscr I}})
$$
for each $N \in {\mathbb N}$, which completes the argument.

It remains to consider the case of trivial group actions. The proof in this case follows along the lines of Kumano-go \cite[Chapter 2, Lemma 2.4]{TKKumanogoBook}. Set
$$
b^{\tau}_{\theta}(y,\eta,x,\xi) = a(y-\tau x,y+(1-\tau)x,\eta+\theta\xi).
$$
For any $k \in {\mathbb N}_0$ there exists a continuous seminorm $|\cdot|_k$ on $S^{(\mu_1,\mu_2,\mu_3)}_{\varrho,\delta}({\mathbb R}^q\times{\mathbb R}^q\times{\mathbb R}^q;E,\tilde{E})_{{\mathscr I}}$ such that, for all $|\alpha'|,|\alpha| \leq k$ and $|\beta'|,|\beta| \leq k$,
\begin{equation}\label{TKbasicestimate}
\begin{aligned}
\|D_x^{\beta'}D_y^{\beta}\partial_{\xi}^{\alpha'}\partial_{\eta}^{\alpha}&b^{\tau}_{\theta}(y,\eta,x,\xi)\|_{{\mathscr I}(E,\tilde{E})} \\
&\lesssim |a|_k \langle y-\tau x \rangle^{\mu_2}\langle y+(1-\tau)x \rangle^{\mu_3}\langle \eta + \theta\xi \rangle^{\mu_1-\varrho(|\alpha|+|\alpha'|)+\delta(|\beta|+|\beta'|)} \\
&\lesssim |a|_k\langle y \rangle^{\mu_2+\mu_3}\langle x \rangle^{|\mu_2|+|\mu_3|}\langle \eta + \theta\xi \rangle^{\mu_1-\varrho(|\alpha|+|\alpha'|)+\delta(|\beta|+|\beta'|)} \\
&\lesssim |a|_k\langle y \rangle^{\mu_2+\mu_3}\langle \langle \eta \rangle^{\delta}x \rangle^{|\mu_2|+|\mu_3|}\langle \eta + \theta\xi \rangle^{\mu_1-\varrho(|\alpha|+|\alpha'|)+\delta(|\beta|+|\beta'|)}.
\end{aligned}
\end{equation}
We are omitting constants in these estimates, therefore writing $\lesssim$ instead of $\leq$. Here and in the sequel it is important to note that, whenever we write $\lesssim$, the implicit constants may depend on $k \in {\mathbb N}_0$, $\vec{\mu}$, the dimension $q$, as well as on $\tau \in {\mathbb R}$, but they are independent of all the variables $(y,\eta,x,\xi)$, the parameter $0 \leq \theta \leq 1$, and the double symbol $a$.

These estimates show, in particular, that $b_{\theta}^{\tau}(y,\eta,x,\xi)$ is an ${\mathscr I}(E,\tilde{E})$-valued amplitude function in $(x,\xi)$. Thus $a^{\tau}_{\theta}(y,\eta)$ is well-defined, and in fact is a $C^{\infty}$-function of $(y,\eta)$ taking values in ${\mathscr I}(E,\tilde{E})$; we can differentiate under the integral sign to find the derivatives of $a_{\theta}^{\tau}$. Taylor expansion shows that
\begin{gather*}
a(y-\tau x,y+(1-\tau)x,\eta+\theta\xi) = \sum_{|\gamma|<N}\frac{\theta^{|\gamma|}}{\gamma !}(\partial^{\gamma}_{\eta}a)(y-\tau x,y+(1-\tau)x,\eta)\xi^{\gamma} \\
+ N\sum_{|\gamma|=N}\frac{\theta^N}{\gamma!}\xi^{\gamma}\int_0^1(1-s)^{N-1}(\partial_{\eta}^{\gamma}a)(y-\tau x,y+(1-\tau)x,\eta+s\theta\xi)\,ds.
\end{gather*}
We insert this in the oscillatory integral formula for $a_{\theta}^{\tau}(y,\eta)$. Using $\xi^{\gamma}e^{-ix\xi} = (-D_x)^{\gamma}e^{-ix\xi}$ and integrating by parts then gives \eqref{TKatauthetaexpansion}.

It remains to show the asserted estimates for $a_{\theta}^{\tau}(y,\eta)$ and for $r_{N,\theta}^{\tau}(y,\eta)$. Because we already have \eqref{TKatauthetaexpansion} we see that both reduce to proving that $\{a^{\tau}_{\theta} : 0 \leq \theta \leq 1\}$ is a family of symbols, where the constants in the symbol estimates can be chosen to be continuous seminorms of $a$ that are independent of $0 \leq \theta \leq 1$. For this proof we may further assume without loss of generality by \eqref{TKatauthetaexpansion} that $\mu_1 < -q$, i.e., the amplitude function $b_{\theta}^{\tau}$ is already integrable with respect to $\xi$ (this uses again that $\delta < \varrho$). Once these uniform estimates are established the continuity of $\theta \mapsto a_{\theta}^{\tau}$ follows from the mean value theorem, using the formula
\begin{gather*}
\partial_{\theta}a_{\theta}^{\tau}(y,\eta) = \sum_{|\alpha+\beta|=1}\tau^{|\alpha|}(1-\tau)^{|\beta|} \cdot \\
\cdot \frac{1}{(2\pi)^q}\iint e^{-ix\xi}((-D_y)^{\alpha}D_{y'}^{\beta}\partial_{\eta}^{\alpha+\beta}a)(y-\tau x,y+(1-\tau)x,\eta+\theta\xi)\,dx\,d\xi
\end{gather*}
which shows that $\{\partial_{\theta}a_{\theta}^{\tau} : 0 \leq \theta \leq 1\}$ is bounded in the symbol topology in view of the uniform symbol estimates.

Because we can differentiate under the integral sign, it is enough to prove that $\|a^{\tau}_{\theta}(y,\eta)\|_{{\mathscr I}(E,\tilde{E})} \lesssim |a|_k\langle y \rangle^{\mu_2+\mu_3}\langle \eta \rangle^{\mu_1}$ for some $k \in {\mathbb N}_0$ as in \eqref{TKbasicestimate}.

We regularize the oscillatory integral for $a_{\theta}^{\tau}$ and obtain
\begin{gather*}
a_{\theta}^{\tau}(y,\eta) = \frac{1}{(2\pi)^q}\iint e^{-ix\xi}c_{\ell,\theta}^{\tau}(y,\eta,x,\xi)\,dx\,d\xi, \\
c_{\ell,\theta}^{\tau}(y,\eta,x,\xi) = \langle\langle \eta \rangle^{\delta}x\rangle^{-2\ell}(1-\langle\eta\rangle^{2\delta}\Delta_{\xi})^{\ell}b_{\theta}^{\tau}(y,\eta,x,\xi)
\end{gather*}
for each $\ell \in {\mathbb N}$. This is integrable in $(x,\xi)$ for $2\ell > |\mu_2|+|\mu_3| + q$ by \eqref{TKbasicestimate}, which we assume henceforth (recall that $\mu_1 < -q$). Let
\begin{align*}
\Omega_1 &= \{\xi : |\xi| \leq \langle \eta \rangle^{\delta}/2\}, \\
\Omega_2 &= \{\xi : \langle \eta \rangle^{\delta}/2 \leq |\xi| \leq \langle \eta \rangle/2\}, \\
\Omega_3 &= \{\xi : |\xi| \geq \langle \eta \rangle/2\},
\end{align*}
and write
$$
a_{\theta}^{\tau}(y,\eta) = \frac{1}{(2\pi)^q}\iint_{{\mathbb R}^{2q}} e^{-ix\xi}c_{\ell,\theta}^{\tau}(y,\eta,x,\xi)\,dx\,d\xi = I_1(y,\eta) + I_2(y,\eta) + I_3(y,\eta),
$$
where
$$
I_j(y,\eta) = \frac{1}{(2\pi)^q}\int_{\Omega_j}\int_{{\mathbb R}^q}e^{-ix\xi}c_{\ell,\theta}^{\tau}(y,\eta,x,\xi)\,dx\,d\xi, \quad j=1,2,3.
$$
We proceed to estimate the three integrals separately. On $\Omega_1 \cup \Omega_2$ we have
$$
\langle \eta \rangle/2 \leq \langle \eta + \theta\xi \rangle \leq (3/2)\langle \eta \rangle,
$$
thus we get from \eqref{TKbasicestimate}
$$
\|c^{\tau}_{\ell,\theta}(y,\eta,x,\xi)\|_{{\mathscr I}(E,\tilde{E})} \lesssim |a|_{2\ell} \langle y \rangle^{\mu_2+\mu_3}\langle \langle \eta \rangle^{\delta} x \rangle^{|\mu_2|+|\mu_3|-2\ell}\langle \eta \rangle^{\mu_1}.
$$
Changing variables $x' = \langle \eta \rangle^{\delta}x$ shows that
$$
\int_{{\mathbb R}^q}\langle \langle \eta \rangle^{\delta} x \rangle^{|\mu_2|+|\mu_3|-2\ell}\,dx \lesssim \langle \eta \rangle^{-\delta q},
$$
and we have
$$
\int_{\Omega_1}\,d\xi \leq \int_{\{\xi : |\xi_j| \leq \langle \eta \rangle^{\delta}/2\}}\,d\xi = \langle \eta \rangle^{\delta q}.
$$
Consequently,
$$
\|I_1(y,\eta)\|_{{\mathscr I}(E,\tilde{E})} \lesssim |a|_{2\ell} \langle y \rangle^{\mu_2+\mu_3}\langle \eta \rangle^{\mu_1}.
$$
In order to estimate $I_2(y,\eta)$, we first note that
$$
\|D_x^{\beta'}c^{\tau}_{\ell,\theta}(y,\eta,x,\xi)\|_{{\mathscr I}(E,\tilde{E})} \lesssim |a|_{2\ell} \langle y \rangle^{\mu_2+\mu_3}\langle \langle \eta \rangle^{\delta} x \rangle^{|\mu_2|+|\mu_3|-2\ell}\langle \eta \rangle^{\mu_1+\delta|\beta'|}
$$
for $|\beta'| \leq 2\ell$ on $\Omega_1 \cup \Omega_2$. This follows from the Leibniz rule, \eqref{TKbasicestimate}, and the estimate
$$
|D_x^{\beta}\langle \langle \eta \rangle^{\delta}x \rangle^{-2\ell}| \lesssim \langle \eta \rangle^{\delta|\beta|}\langle \langle \eta \rangle^{\delta}x \rangle^{-2\ell}.
$$
On $\Omega_2$ we write
$$
\int_{{\mathbb R}^q}e^{-ix\xi}c_{\ell,\theta}^{\tau}(y,\eta,x,\xi)\,dx = |\xi|^{-2\ell}\int_{{\mathbb R}^q}e^{-ix\xi}(-\Delta_x)^{\ell}c_{\ell,\theta}^{\tau}(y,\eta,x,\xi)\,dx.
$$
Thus
$$
\|I_2(y,\eta)\|_{{\mathscr I}(E,\tilde{E})} \lesssim |a|_{2\ell}\langle y \rangle^{\mu_2+\mu_3}\langle \eta \rangle^{\mu_1+2\ell \delta} \underbrace{\int_{{\mathbb R}^q}\langle \langle \eta \rangle^{\delta} x \rangle^{|\mu_2|+|\mu_3|-2\ell}\,dx}_{\lesssim \langle \eta \rangle^{-\delta q}} \, \int_{\Omega_2}|\xi|^{-2\ell}\,d\xi.
$$
To estimate the $\xi$-integral we pass to polar coordinates and obtain
$$
\int_{\Omega_2}|\xi|^{-2\ell}\,d\xi \lesssim \int_{\langle \eta \rangle^{\delta}/2}^{\infty}r^{q-1-2\ell}\,dr \lesssim \langle \eta \rangle^{(q-2\ell)\delta},
$$
and consequently
$$
\|I_2(y,\eta)\|_{{\mathscr I}(E,\tilde{E})} \lesssim |a|_{2\ell}\langle y \rangle^{\mu_2+\mu_3}\langle \eta \rangle^{\mu_1}.
$$
It remains to estimate $I_3(y,\eta)$. On $\Omega_3$ we have $\langle \eta + \theta\xi \rangle \leq \langle \eta \rangle + |\xi| \leq 3|\xi|$, and consequently
$$
\|D_x^{\beta'}c^{\tau}_{\ell,\theta}(y,\eta,x,\xi)\|_{{\mathscr I}(E,\tilde{E})} \lesssim |a|_k \langle y \rangle^{\mu_2+\mu_3}\langle \langle \eta \rangle^{\delta} x \rangle^{|\mu_2|+|\mu_3|-2\ell}|\xi|^{\delta(|\beta'|+2\ell)}
$$
for $|\beta'| \leq k$, where $k \geq 2\ell$. As above, we write
$$
\int_{{\mathbb R}^q}e^{-ix\xi}c_{\ell,\theta}^{\tau}(y,\eta,x,\xi)\,dx = |\xi|^{-2L}\int_{{\mathbb R}^q}e^{-ix\xi}(-\Delta_x)^{L}c_{\ell,\theta}^{\tau}(y,\eta,x,\xi)\,dx
$$
on $\Omega_3$, where $L \geq \ell$ is chosen large enough, to be determined momentarily. Thus
$$
\|I_3(y,\eta)\|_{{\mathscr I}(E,\tilde{E})} \lesssim |a|_{2L}\langle y \rangle^{\mu_2+\mu_3}\underbrace{\int_{{\mathbb R}^q}\langle \langle \eta \rangle^{\delta} x \rangle^{|\mu_2|+|\mu_3|-2\ell}\,dx}_{\lesssim \langle \eta \rangle^{-\delta q} \leq 1} \, \int_{\Omega_3}|\xi|^{-2L(1-\delta)+2\ell\delta}\,d\xi.
$$
We have
$$
\int_{\Omega_3}|\xi|^{-2L(1-\delta)+2\ell\delta}\,d\xi \lesssim \int_{\langle \eta \rangle/2}^{\infty}r^{-2L(1-\delta)+2\ell\delta+q-1}\,dr \lesssim \langle \eta \rangle^{-2L(1-\delta)+2\ell\delta + q},
$$
and consequently
$$
\|I_3(y,\eta)\|_{{\mathscr I}(E,\tilde{E})} \lesssim |a|_{2L}\langle y \rangle^{\mu_2+\mu_3}\langle \eta \rangle^{-2L(1-\delta)+2\ell\delta + q} \lesssim |a|_{2L}\langle y \rangle^{\mu_2+\mu_3}\langle \eta \rangle^{\mu_1}
$$
provided that $L \geq \ell$ is chosen so large that $-2L(1-\delta)+2\ell\delta + q \leq \mu_1$ (recall that $\mu_1 < -q$). This completes the proof of the lemma.
\end{proof}

\begin{definition}
For $\vec{\mu} \in {\mathbb R}^3$ let
$$
S^{\vec{\mu}}_{\varrho,\delta}({\mathbb R}^q\times{\mathbb R}^q\times{\mathbb R}^q;(E,\kappa),(\tilde{E},\tilde{\kappa}))_{{\mathscr I},(0)}
$$
be the closure of $C_c^{\infty}({\mathbb R}^q\times{\mathbb R}^q\times{\mathbb R}^q,{\mathscr I}(E,\tilde{E}))$ in the symbol space $S^{\vec{\mu}}_{\varrho,\delta}({\mathbb R}^q\times{\mathbb R}^q\times{\mathbb R}^q;(E,\kappa),(\tilde{E},\tilde{\kappa}))_{{\mathscr I}}$. We use the analogous notation for the classes of simple symbols.
\end{definition}

\begin{lemma}\label{TKdensitylemma}
We have
$$
S^{\vec{\mu}}_{\varrho,\delta}({\mathbb R}^q\times{\mathbb R}^q\times{\mathbb R}^q;(E,\kappa),(\tilde{E},\tilde{\kappa}))_{{\mathscr I}} \subset 
S^{\vec{\mu}'}_{\varrho,\delta}({\mathbb R}^q\times{\mathbb R}^q\times{\mathbb R}^q;(E,\kappa),(\tilde{E},\tilde{\kappa}))_{{\mathscr I},(0)}
$$
if $\mu_j < \mu_j'$ for all $j=1,2,3$. The analogous result holds for the classes of simple symbols.
\end{lemma}
\begin{proof}
This follows by the standard argument: Let $\chi \in C^{\infty}({\mathbb R}^{3q})$ be an excision function of the origin, i.e., $\chi \equiv 0$ near $(0,0,0) \in {\mathbb R}^{3q}$, and $\chi \equiv 1$ for large $|y,y',\eta|$. More specifically, assume that $\chi(y,y',\eta) \equiv 0$ for $|y,y',\eta| \leq 1$ and $\chi(y,y',\eta) \equiv 1$ for $|y,y',\eta| \geq 2$. For $a \in S^{\vec{\mu}}_{\varrho,\delta}({\mathbb R}^q\times{\mathbb R}^q\times{\mathbb R}^q;(E,\kappa),(\tilde{E},\tilde{\kappa}))_{{\mathscr I}}$ and $j \in {\mathbb N}$ define
$$
a_j(y,y',\eta) = (1-\chi(y/j,y'/j,\eta/j))a(y,y',\eta) \in C_c^{\infty}({\mathbb R}^q\times{\mathbb R}^q\times{\mathbb R}^q,{\mathscr I}(E,\tilde{E})).
$$
Then
\begin{align*}
(a - a_j)(y,y',\eta) &= \chi(y/j,y'/j,\eta/j)a(y,y',\eta) \\
&\in S^{\vec{\mu}'}_{\varrho,\delta}({\mathbb R}^q\times{\mathbb R}^q\times{\mathbb R}^q;(E,\kappa),(\tilde{E},\tilde{\kappa}))_{{\mathscr I}},
\end{align*}
and $a - a_j \to 0$ as $j \to \infty$ in $S^{\vec{\mu}'}_{\varrho,\delta}({\mathbb R}^q\times{\mathbb R}^q\times{\mathbb R}^q;(E,\kappa),(\tilde{E},\tilde{\kappa}))_{{\mathscr I}}$ which proves the claim.

To see this convergence, set $\varepsilon = \min\{\mu_j'-\mu_j : j=1,2,3\} > 0$ and use the Leibniz rule and basic inequalities to estimate
\begin{align*}
\langle y \rangle^{-\mu_2'}\langle y' \rangle^{-\mu_3'}&\langle \eta \rangle^{-\mu_1'+\varrho|\alpha_3|-\delta(|\alpha_1|+|\alpha_2|)}\cdot \\
\cdot&\|\tilde{\kappa}^{-1}_{\langle \eta \rangle}[D^{\alpha_1}_yD^{\alpha_2}_{y'}\partial^{\alpha_3}_{\eta}[\chi(y/j,y'/j,\eta/j)a(y,y',\eta)]]\kappa_{\langle \eta \rangle}\|_{{\mathscr I}}  \\
\leq \langle y,y',\eta \rangle^{-\varepsilon}\cdot
&\sum_{\substack{\gamma_1\leq\alpha_1 \\ \gamma_2 \leq \alpha_2 \\ \gamma_3 \leq \alpha_3}}\binom{\alpha_1}{\gamma_1}\binom{\alpha_2}{\gamma_2}\binom{\alpha_3}{\gamma_3} \cdot \\
\cdot \langle y \rangle^{-\mu_2}\langle y' \rangle^{-\mu_3} &\langle \eta \rangle^{-\mu_1+\varrho|\gamma_3| -\delta(|\gamma_1|+|\gamma_2|)} \|\tilde{\kappa}^{-1}_{\langle \eta \rangle}[D^{\gamma_1}_yD^{\gamma_2}_{y'}\partial^{\gamma_3}_{\eta}a(y,y',\eta)]\kappa_{\langle \eta \rangle}\|_{{\mathscr I}} \cdot \\
\cdot |(D^{\alpha_1-\gamma_1}_yD^{\alpha_2-\gamma_2}_{y'}\partial^{\alpha_3-\gamma_3}_{\eta}&\chi)(y/j,y'/j,\eta/j)|
\langle y/j,y'/j,\eta/j \rangle^{|\alpha_1-\gamma_1|+|\alpha_2-\gamma_2|+|\alpha_3-\gamma_3|}.
\end{align*}
On the support of $(D^{\alpha_1-\gamma_1}_yD^{\alpha_2-\gamma_2}_{y'}\partial^{\alpha_3-\gamma_3}_{\eta}\chi)(y/j,y'/j,\eta/j)$ we have
$$
\langle y/j,y'/j,\eta/j \rangle^{|\alpha_1-\gamma_1|+|\alpha_2-\gamma_2|+|\alpha_3-\gamma_3|} \leq 3^{|\alpha_1|+|\alpha_2|+|\alpha_3|}
$$
and $\langle y,y',\eta \rangle^{-\varepsilon} \leq j^{-\varepsilon}$ by the choice of $\chi$. Summing up, this shows that for every $\alpha_1,\alpha_2,\alpha_3 \in {\mathbb N}_0^q$ there exists a continuous seminorm $|\cdot|$ on $S_{\varrho,\delta}^{\vec{\mu}}({\mathbb R}^q\times{\mathbb R}^q\times{\mathbb R}^q;(E,\kappa),(\tilde{E},\tilde{\kappa}))_{\mathscr I}$ such that
\begin{align*}
\sup\{&\langle y \rangle^{-\mu_2'}\langle y' \rangle^{-\mu_3'}\langle \eta \rangle^{-\mu_1'+\varrho|\alpha_3|-\delta(|\alpha_1|+|\alpha_2|)}\cdot \\
\cdot&\|\tilde{\kappa}^{-1}_{\langle \eta \rangle}[D^{\alpha_1}_yD^{\alpha_2}_{y'}\partial^{\alpha_3}_{\eta}[\chi(y/j,y'/j,\eta/j)a(y,y',\eta)]]\kappa_{\langle \eta \rangle}\|_{{\mathscr I}}  : y,y',\eta \in {\mathbb R}^q\} 
\leq  \frac{1}{j^{\varepsilon}} |a|,
\end{align*}
thus proving the claimed convergence.
\end{proof}

\begin{lemma}\label{TKTechnicalLemma2}
For every $\tau \in {\mathbb R}$ the map
$$
T_{\tau} : S^{\vec{\mu}}_{\varrho,\delta}({\mathbb R}^q\times{\mathbb R}^q;(E,\kappa),(\tilde{E},\tilde{\kappa}))_{{\mathscr I}} \to S^{\vec{\mu}}_{\varrho,\delta}({\mathbb R}^q\times{\mathbb R}^q;(E,\kappa),(\tilde{E},\tilde{\kappa}))_{{\mathscr I}}
$$
given by
$$
T_{\tau}(a)(y,\eta) = \frac{1}{(2\pi)^q}\iint e^{-ix\xi}a(y-\tau x,\eta+\xi)\,dx\,d\xi
$$
is a topological isomorphism with inverse $T_{-\tau}$. We have
$$
T_{\tau}(a) (y,\eta) \sim \sum_{\alpha \in {\mathbb N}_0^q}\frac{(-\tau)^{|\alpha|}}{\alpha !}D_y^{\alpha}\partial_{\eta}^{\alpha}a(y,\eta).
$$
\end{lemma}
\begin{proof}
The continuity of $T_{\tau}$ and the asymptotic expansion follow from Lemma~\ref{TKTechnicalLemma}. It remains to show that $T_{\tau}$ is invertible with inverse $T_{-\tau}$. It is easy to see that
$$
T_{\tau} : {\mathscr S}({\mathbb R}^q\times{\mathbb R}^q,{\mathscr I}(E,\tilde{E})) \to {\mathscr S}({\mathbb R}^q\times{\mathbb R}^q,{\mathscr I}(E,\tilde{E}))
$$
is invertible with inverse $T_{-\tau}$. By continuity and Lemma~\ref{TKdensitylemma} we then get that
$$
T_{\tau} : S^{\vec{\mu}+\vec{\varepsilon}}_{\varrho,\delta}({\mathbb R}^q\times{\mathbb R}^q;(E,\kappa),(\tilde{E},\tilde{\kappa}))_{{\mathscr I},(0)} \to S^{\vec{\mu}+\vec{\varepsilon}}_{\varrho,\delta}({\mathbb R}^q\times{\mathbb R}^q;(E,\kappa),(\tilde{E},\tilde{\kappa}))_{{\mathscr I},(0)}
$$
is invertible with inverse $T_{-\tau}$, where $\vec{\varepsilon} = (\varepsilon,\varepsilon)$ for some $\varepsilon > 0$. Using again Lemma~\ref{TKdensitylemma} this implies that $T_{\tau}$ is invertible with inverse $T_{-\tau}$ on the space $S^{\vec{\mu}}_{\varrho,\delta}({\mathbb R}^q\times{\mathbb R}^q;(E,\kappa),(\tilde{E},\tilde{\kappa}))_{{\mathscr I}}$ as desired.
\end{proof}

\begin{theorem}\label{TKOperatorClassIndependentofTau}
\begin{enumerate}
\item Let $a(y,y',\eta) \in S^{(\mu_1,\mu_2,\mu_3)}_{\varrho,\delta}({\mathbb R}^q\times{\mathbb R}^q\times{\mathbb R}^q;(E,\kappa),(\tilde{E},\tilde{\kappa}))_{{\mathscr I}}$. Then
$$
\Op(a) = \Op_{\tau}(b) : {\mathscr S}({\mathbb R}^q,E) \to {\mathscr S}({\mathbb R}^q,\tilde{E}),
$$
where $b \in S^{(\mu_1,\mu_2+\mu_3)}_{\varrho,\delta}({\mathbb R}^q\times{\mathbb R}^q,(E,\kappa),(\tilde{E},\tilde{\kappa}))_{{\mathscr I}}$ is given by
$$
b(y,\eta) = \frac{1}{(2\pi)^q}\iint e^{-ix\xi}a(y-\tau x,y+(1-\tau)x,\eta+\xi)\,dx\,d\xi.
$$
We have the asymptotic expansion
$$
b(y,\eta) \sim \sum_{\alpha,\beta \in {\mathbb N}_0^q}\frac{1}{\alpha !\beta !}\tau^{|\alpha|}(1-\tau)^{|\beta|}\partial_{\eta}^{\alpha+\beta}(-D_y)^{\alpha}D_{y'}^{\beta}a(y,y',\eta)\big|_{y'=y}.
$$
\item The class
\begin{align*}
\Psi^{\vec{\mu}}_{\varrho,\delta}({\mathbb R}^q;(E,\kappa),(\tilde{E},\kappa))_{{\mathscr I}} = \{\Op_{\tau}(a) &: {\mathscr S}({\mathbb R}^q,E) \to {\mathscr S}({\mathbb R}^q,\tilde{E}) : \\
&a \in S^{\vec{\mu}}_{\varrho,\delta}({\mathbb R}^q\times{\mathbb R}^q;(E,\kappa),(\tilde{E},\kappa))_{{\mathscr I}}\}
\end{align*}
is independent of $\tau \in {\mathbb R}$, and the map
$$
S^{\vec{\mu}}_{\varrho,\delta}({\mathbb R}^q\times{\mathbb R}^q;(E,\kappa),(\tilde{E},\tilde{\kappa}))_{{\mathscr I}} \ni a \mapsto \Op_{\tau}(a) \in \Psi^{\vec{\mu}}_{\varrho,\delta}({\mathbb R}^q;(E,\kappa),(\tilde{E},\tilde{\kappa}))_{{\mathscr I}}
$$
is a bijection. The operator space $\Psi^{\vec{\mu}}_{\varrho,\delta}({\mathbb R}^q;(E,\kappa),(\tilde{E},\tilde{\kappa}))_{{\mathscr I}}$ carries a Fr{\'e}chet topology transferred from the symbol topology by any such $\tau$-quantization map. This topology is independent of $\tau$.
\end{enumerate}
\end{theorem}
\begin{proof}
The symbol $b(y,\eta)$ given by the oscillatory integral in the statement of Part (1) belongs to $S^{(\mu_1,\mu_2+\mu_3)}_{\varrho,\delta}({\mathbb R}^q\times{\mathbb R}^q,(E,\kappa),(\tilde{E},\tilde{\kappa}))_{{\mathscr I}}$ and has the asserted asymptotic expansion by Lemma~\ref{TKTechnicalLemma}. Thus it remains to show that
$$
\Op(a) = \Op_{\tau}(b) : {\mathscr S}({\mathbb R}^q,E) \to {\mathscr S}({\mathbb R}^q,\tilde{E}).
$$
Note that this is certainly the case whenever $a(y,y',\eta) \in {\mathscr S}({\mathbb R}^q\times{\mathbb R}^q\times{\mathbb R}^q,{\mathscr I}(E,\tilde{E}))$ by an elementary argument. Now let $a(y,y',\eta)$ be arbitrary as in the statement of the theorem. Choose $\varepsilon > 0$ and a sequence $a_j \in {\mathscr S}({\mathbb R}^q\times{\mathbb R}^q\times{\mathbb R}^q,{\mathscr I}(E,\tilde{E}))$ such that $a_j \to a$ in $S^{(\mu_1+\varepsilon,\mu_2+\varepsilon,\mu_3+\varepsilon)}_{\varrho,\delta}({\mathbb R}^q\times{\mathbb R}^q\times{\mathbb R}^q;(E,\kappa),(\tilde{E},\tilde{\kappa}))_{{\mathscr I}}$. This is possible by Lemma~\ref{TKdensitylemma}. Let $b_j \in {\mathscr S}({\mathbb R}^q\times{\mathbb R}^q,{\mathscr I}(E,\tilde{E}))$ be associated to $a_j$ by the oscillatory integral formula stated in the theorem. Then
$$
b_j \to b \in S^{(\mu_1+\varepsilon,\mu_2+\mu_3+2\varepsilon)}({\mathbb R}^q\times{\mathbb R}^q;(E,\kappa),(\tilde{E},\tilde{\kappa}))_{{\mathscr I}}
$$
by Lemma~\ref{TKTechnicalLemma}. For any $u \in {\mathscr S}({\mathbb R}^q,E)$ we then obtain, with convergence in ${\mathscr S}({\mathbb R}^q,\tilde{E})$,
$$
\Op(a)u \longleftarrow \Op(a_j)u = \Op_{\tau}(b_j)u \longrightarrow \Op_{\tau}(b)u,
$$
proving Part (1).

The injectivity of the Kohn-Nirenberg quantization map $a \mapsto \Op_0(a)$ is standard, see \cite{TKKumanogoBook}. It remains to note that, by Part (1) of the theorem, $\Op_0(a) = \Op_{\tau}(b)$, where $b = T_{\tau}(a)$ and $a = T_{-\tau}b$ with the isomorphisms $T_{\tau}$ and $T_{-\tau}$ from Lemma~\ref{TKTechnicalLemma2}. This proves all claims in Part (2).
\end{proof}

As expected, composition of operators is well behaved for the twisted calculus. Typical examples for the three Banach operator ideals that appear in the statement of Theorem~\ref{TKCompositionTheorem} below are ${\mathscr I}_1={\mathscr I}_3={\mathscr I}$ and ${\mathscr I}_2 = {\mathscr L}$, or ${\mathscr I}_2={\mathscr I}_3={\mathscr I}$ and ${\mathscr I}_1 = {\mathscr L}$, where ${\mathscr I}$ is some fixed Banach operator ideal, but also ${\mathscr I}_1 = {\mathscr C}_p$, ${\mathscr I}_2 = {\mathscr C}_{p'}$, and ${\mathscr I}_3 = {\mathscr C}_1$, where $1 < p,p' < \infty$ with $\frac{1}{p}+\frac{1}{p'} = 1$.

\begin{theorem}[Composition]\label{TKCompositionTheorem}
Let ${\mathscr I}_j$, $j=1,2,3$, be Banach operator ideals such that composition of operators is continuous in ${\mathscr I}_1 \times {\mathscr I}_2 \to {\mathscr I}_3$. Let
\begin{align*}
A &\in \Psi^{\vec{\mu}_1}_{\varrho,\delta}({\mathbb R}^q;(\tilde{E},\tilde{\kappa}),(\hat{E},\hat{\kappa}))_{{\mathscr I}_1}, \\
B &\in \Psi^{\vec{\mu}_2}_{\varrho,\delta}({\mathbb R}^q;(E,\kappa),(\tilde{E},\tilde{\kappa}))_{{\mathscr I}_2}.
\end{align*}
Then $A\circ B \in \Psi^{\vec{\mu}_1+\vec{\mu}_2}_{\varrho,\delta}({\mathbb R}^q;(E,\kappa),(\hat{E},\hat{\kappa}))_{{\mathscr I}_3}$,
and the map $(A,B) \mapsto A\circ B$ is bilinear and continuous in the indicated operator space topologies.

If $A = \Op_0(a)$ and $B = \Op_0(b)$ then $A\circ B = \Op_0(a{\#}b)$, where the Leibniz product
$$
(a{\#}b)(y,\eta) \in S^{\vec{\mu}_1+\vec{\mu}_2}_{\varrho,\delta}({\mathbb R}^q\times{\mathbb R}^q;(E,\kappa),(\hat{E},\hat{\kappa}))_{{\mathscr I}_3}
$$
has the asymptotic expansion
$$
a{\#}b \sim \sum_{\alpha \in {\mathbb N}_0^q} \frac{1}{\alpha!} (\partial^{\alpha}_{\eta}a)(D_y^{\alpha}b).
$$
\end{theorem}
\begin{proof}
By Theorem~\ref{TKOperatorClassIndependentofTau} write $A = \Op_0(a)$ and $B = \Op_0(b) = \Op_1(T_1(b))$ with $T_1$ from Lemma~\ref{TKTechnicalLemma2}. Then
$$
A \circ B = \Op(c), \quad c(y,y',\eta) = a(y,\eta)T_1(b)(y',\eta).
$$
By Theorem~\ref{TKOperatorClassIndependentofTau} again, $A \circ B = \Op_0(a\#b)$, and all assertions of the theorem follow.
\end{proof}

Let $E_0$ and $E_1$ be separable complex Hilbert spaces, and let
$$
[\cdot,\cdot] : E_0\times E_1 \to {\mathbb C}
$$
be a sesquilinear map that satisfies the following properties:
\begin{enumerate}
\item (Continuity) There exists a constant $C > 0$ such that
$$
|[e_0,e_1]| \leq C \|e_0\|_{E_0}\|e_1\|_{E_1}
$$
for all $e_0 \in E_0$ and all $e_1 \in E_1$.
\item (Nondegeneracy) There exists a constant $c > 0$ such that
$$
\sup\{|[e_0,e_1]| : \|e_0\|_{E_0} \leq 1\} \geq c\|e_1\|_{E_1}
$$
for all $e_1 \in E_1$, and such that
$$
\sup\{|[e_0,e_1]| : \|e_1\|_{E_1} \leq 1\} \geq c\|e_0\|_{E_0}
$$
for all $e_0 \in E_0$.
\end{enumerate}
As can be readily seen, these conditions on $[\cdot,\cdot]$ are equivalent to the existence of a topological isomorphism $J : E_1 \to E_0$ such that
$$
[e_0,e_1] = \langle e_0,Je_1 \rangle_{E_0}
$$
for all $e_0 \in E_0$ and $e_1 \in E_1$.

Now let $E_0$ and $E_1$ be Hilbert spaces equipped with such a nondegenerate continuous sesquilinear pairing $[\cdot,\cdot] : E_0 \times E_1 \to {\mathbb C}$, and let $\tilde{E}_0$ and $\tilde{E}_1$ be Hilbert spaces equipped with the nondegenerate continuous sesquilinear pairing $[\cdot,\cdot]_{\sim} : \tilde{E}_0 \times \tilde{E}_1 \to {\mathbb C}$. Any bounded operator $G : E_0 \to \tilde{E}_0$ then has an adjoint with respect to these pairings, the bounded operator $G^{\#} : \tilde{E}_1 \to E_1$ that is defined by the relation
$$
[Ge_0,\tilde{e}_1]_{\sim} = [e_0,G^{\#}\tilde{e}_1]
\quad \textup{for $e_0 \in E_0$ and $\tilde{e}_1 \in \tilde{E}_1$.}
$$
If $[e_0,e_1] = \langle e_0,Je_1 \rangle_{E_0}$ and $[\tilde{e}_0,\tilde{e}_1]_{\sim} = \langle \tilde{e}_0,\tilde{J}\tilde{e}_1 \rangle_{\tilde{E}_0}$ with $J : E_1 \to E_0$ and $\tilde{J} : \tilde{E}_1 \to \tilde{E}_0$ as above, then $G^{\#} = J^{-1}G^{*}\tilde{J}$, where $G^{*} \in {\mathscr L}(\tilde{E}_0,E_0)$ is the Hilbert space adjoint to $G$. In particular, if $G \in {\mathscr I}(E_0,\tilde{E}_0)$, and the operator ideal ${\mathscr I}$ is closed under taking Hilbert space adjoints, then $G^{\#} \in {\mathscr I}(\tilde{E}_1,E_1)$.

Finally, if $\kappa_{\varrho} : E_0 \to E_0$, $\varrho > 0$, is a strongly continuous group action, then $[\kappa^{\#}]^{-1}$, which as indicated by the notation is defined as $[\kappa_{\varrho}^{\#}]^{-1} : E_1 \to E_1$, $\varrho > 0$, is a strongly continuous group action on $E_1$ (this follows from the strong continuity of the group of Hilbert space adjoints $\kappa_{\varrho}^*$, see \cite[Chapter I.5.b]{TKEngelNagel}).

\begin{theorem}[Adjoints]\label{TKAdjointTheorem}
Let ${\mathscr I}$ be a Banach operator ideal that is closed under taking Hilbert space adjoints, and let $A \in \Psi^{\vec{\mu}}_{\varrho,\delta}({\mathbb R}^q;(E_0,\kappa),(\tilde{E}_0,\tilde{\kappa}))_{{\mathscr I}}$.

Let $[\cdot,\cdot] : E_0\times E_1 \to {\mathbb C}$ and $[\cdot,\cdot]_{\sim} : \tilde{E}_0\times\tilde{E}_1 \to {\mathbb C}$ be continuous nondegenerate sesquilinear pairings.

Then the formal adjoint operator $A^{\#} : {\mathscr S}({\mathbb R}^q,\tilde{E}_1) \to {\mathscr S}({\mathbb R}^q,E_1)$, defined by the relation
$$
\int_{{\mathbb R}^q} [Au(y),v(y)]_{\sim}\,dy = \int_{{\mathbb R}^q} [u(y),A^{\#}v(y)]\,dy
$$
for $u \in {\mathscr S}({\mathbb R}^q,E_0)$ and $v \in {\mathscr S}({\mathbb R}^q,\tilde{E}_1)$, is well-defined, and we have
$$
A^{\#} \in \Psi^{\vec{\mu}}_{\varrho,\delta}({\mathbb R}^q;(\tilde{E}_1,[\tilde{\kappa}^{\#}]^{-1}),(E_1,[\kappa^{\#}]^{-1}))_{{\mathscr I}}.
$$
If $A = \Op_0(a)$ then $A^{\#} = \Op_0(b)$, where $b(y,\eta)$ has the asymptotic expansion
$$
b(y,\eta) \sim \sum_{\alpha \in {\mathbb N}_0^q} \frac{1}{\alpha !} D^{\alpha}_{y}\partial_{\eta}^{\alpha}a(y,\eta)^{\#}.
$$
\end{theorem}
\begin{proof}
Consider first the case that $E_1=E_0$ and $\tilde{E}_1=\tilde{E}_0$, and that both pairings are merely the inner products. In this case $A^{\#}$ is the standard formal adjoint $A^*$, and it is evident that
$A^* = \Op_1(c)$, where $c(y,\eta) = a(y,\eta)^* : \tilde{E}_0 \to E_0$ is the Hilbert space adjoint for each $(y,\eta) \in {\mathbb R}^q\times{\mathbb R}^q$. Note that
$$
c(y,\eta) \in S^{\vec{\mu}}_{\varrho,\delta}({\mathbb R}^q\times{\mathbb R}^q;(\tilde{E}_0,[\tilde{\kappa}^*]^{-1}),(E_0,[\kappa^*]^{-1}))_{{\mathscr I}}.
$$
This follows from the symbol estimates for $a(y,\eta)$, the properties of the Hilbert space adjoint, and the continuity of $* : {\mathscr I}(E_0,\tilde{E}_0) \to {\mathscr I}(\tilde{E}_0,E_0)$ in the ${\mathscr I}$-norm, where the latter is a consequence of the closed graph theorem. Consequently, all assertions follow from Theorem~\ref{TKOperatorClassIndependentofTau} in this case.

The general case follows from $A^{\#} = \Op_0(J^{-1}) \circ A^* \circ \Op_0(\tilde{J})$ with the isomorphisms $J : E_1 \to E_0$ and $\tilde{J} : \tilde{E}_1 \to \tilde{E}_0$ associated with the sesquilinear pairings discussed above. Note that
\begin{align*}
\tilde{J} &\in S_{\varrho,\delta}^{(0,0)}({\mathbb R}^q\times{\mathbb R}^q;(\tilde{E}_1,[\tilde{\kappa}^{\#}]^{-1}),(\tilde{E}_0,[\tilde{\kappa}^*]^{-1})), \\
J^{-1} &\in S_{\varrho,\delta}^{(0,0)}({\mathbb R}^q\times{\mathbb R}^q;(E_0,[\kappa^{*}]^{-1}),(E_1,[\kappa^{\#}]^{-1})),
\end{align*}
and both are independent of $(y,\eta)$.
\end{proof}


\section{Boundedness and compactness}\label{TKsec-CompactOperators}

We remind the reader about our standing assumption that $0 \leq \delta < \varrho \leq 1$.

\begin{theorem}\label{TKBoundednessTheorem}
Every $A \in \Psi^{\vec{\mu}}_{\varrho,\delta}({\mathbb R}^q;(E,\kappa),(\tilde{E},\tilde{\kappa}))$ extends to a bounded operator
$$
A : {\mathcal W}^{\vec{s}}({\mathbb R}^q,E) \to {\mathcal W}^{\vec{s}-\vec{\mu}}({\mathbb R}^q,\tilde{E})
$$
for all $\vec{s} \in {\mathbb R}^2$, and the induced map
$$
\Psi^{\vec{\mu}}_{\varrho,\delta}({\mathbb R}^q;(E,\kappa),(\tilde{E},\tilde{\kappa})) \to
{\mathscr L}({\mathcal W}^{\vec{s}}({\mathbb R}^q,E),{\mathcal W}^{\vec{s}-\vec{\mu}}({\mathbb R}^q,\tilde{E}))
$$
is continuous.
\end{theorem}
\begin{proof}
We first observe that, by Theorem~\ref{TKCompositionTheorem}, it suffices to consider only the case that $\vec{s} = \vec{\mu} = (0,0)$. Otherwise we replace $A$ by
$$
B = \langle D_y \rangle^{s_1-\mu_1}\circ\langle y \rangle^{s_2-\mu_2}\circ A \circ \langle y \rangle^{-s_2}\circ\langle D_y \rangle^{-s_1} \in \Psi^{(0,0)}_{\varrho,\delta}({\mathbb R}^q;(E,\kappa),(\tilde{E},\tilde{\kappa}));
$$
note that the map $A \mapsto B$ is continuous in the operator space topologies. Moreover, the proof of the theorem reduces further to the case of trivial group actions $\kappa_{\varrho} \equiv \textup{Id}_E$ and $\tilde{\kappa}_{\varrho} \equiv \textup{Id}_{\tilde{E}}$. To see this let $S$ and $T$ be the operators associated with $(E,\kappa)$ from \eqref{TKSTOperators}, and let $\tilde{S}$ and $\tilde{T}$ be the ones associated with $(\tilde{E},\tilde{\kappa})$, respectively. By Proposition~\ref{TKBddnessWsScale} we need to prove that
$$
\tilde{S}AT : L^2({\mathbb R}^q,\ell^2({\mathbb N}_0,E)) \to L^2({\mathbb R}^q,\ell^2({\mathbb N}_0,\tilde{E}))
$$
is continuous. By Theorem~\ref{TKCompositionTheorem} we have
$$
\tilde{S}AT \in \Psi^{(0,0)}_{\varrho,\delta}({\mathbb R}^q,\ell^2({\mathbb N}_0,E),\ell^2({\mathbb N}_0,\tilde{E})),
$$
and the map $A \mapsto \tilde{S}AT$ is continuous. This effectively eliminates the group actions, and we may therefore assume from the beginning that the group actions on both Hilbert spaces $E$ and $\tilde{E}$ are trivial.

The remaining case of $\vec{\mu}=\vec{s}=(0,0)$ and trivial group actions, however, is standard. In view of $\delta < \varrho$ it follows from H{\"o}rmander's elegant argument of the $L^2$-boundedness of basic pseudodifferential operators, see \cite[Theorem 18.1.11]{TKHormanderVol3}, which we proceed to outline for the sake of completeness.

Let first $A = \Op_0(a)$ with $a(y,\eta) \in S^{(-q-1,0)}_{\varrho,\delta}({\mathbb R}^q\times{\mathbb R}^q;E,\tilde{E})$. Then
$$
Au(y) = \int_{{\mathbb R}^q} k(y,y')u(y')\,dy', \; u \in C_c^{\infty}({\mathbb R}^q,E),
$$
where
$$
k(y,y') = (2\pi)^{-q}\int_{{\mathbb R}^q} e^{i(y-y')\eta}a(y,\eta)\,d\eta \in {\mathscr L}(E,\tilde{E}).
$$
The function $k(y,y')$ is continuous on ${\mathbb R}^q\times{\mathbb R}^q$, and
$$
\|k(y,y')\|_{{\mathscr L}(E,\tilde{E})} \leq (2\pi)^{-q}\int_{{\mathbb R}^q}\|a(y,\eta)\|_{{\mathscr L}(E,\tilde{E})}\,d\eta \leq C_q \cdot |a|_{S_{\varrho,\delta}^{-q-1}}
$$
for a continuous seminorm $|\cdot|_{S_{\varrho,\delta}^{-q-1}}$ on $S^{(-q-1,0)}_{\varrho,\delta}({\mathbb R}^q\times{\mathbb R}^q;E,\tilde{E})$. In view of
$$
(y-y')^{\alpha}k(y,y') = \frac{(-1)^{|\alpha|}}{(2\pi)^q}\int_{{\mathbb R}^q}e^{i(y-y')\eta}D^{\alpha}_{\eta}a(y,\eta)\,d\eta
$$
for all $\alpha \in {\mathbb N}_0^q$ we see that there exists a continuous seminorm $|\cdot|$ on $S^{(-q-1,0)}_{\varrho,\delta}({\mathbb R}^q\times{\mathbb R}^q;E,\tilde{E})$ such that
$$
(1 + |y-y'|)^{q+1}\|k(y,y')\|_{{\mathscr L}(E,\tilde{E})} \leq |a|.
$$
In particular,
$$
\sup_{y \in {\mathbb R}^q} \int_{{\mathbb R}^q} \|k(y,y')\|_{{\mathscr L}(E,\tilde{E})}\,dy' \leq C_q |a|, \quad
\sup_{y' \in {\mathbb R}^q} \int_{{\mathbb R}^q} \|k(y,y')\|_{{\mathscr L}(E,\tilde{E})}\,dy \leq C_q |a|,
$$
and Schur's Lemma therefore implies the $L^2$-continuity of $A$ with operator norm bounded by $C_q|a|$.

We next prove the $L^2$-continuity of all operators of class $\Psi^{(\mu,0)}_{\varrho,\delta}({\mathbb R}^q;E,\tilde{E})$ for any $\mu < 0$. This follows inductively by considering $\mu_j = -(q+1)/2^{j}$ for $j \in {\mathbb N}_0$. The case $j=0$ was just discussed. Generally, if $A \in \Psi^{(\mu_{j+1},0)}_{\varrho,\delta}({\mathbb R}^q;E,\tilde{E})$, then $A^*A \in \Psi_{\varrho,\delta}^{(\mu_j,0)}({\mathbb R}^q;E,E)$ by Theorems~\ref{TKCompositionTheorem} and \ref{TKAdjointTheorem}, and thus
$$
\|Au\|^2_{L^2({\mathbb R}^q,\tilde{E})} = \langle Au,Au \rangle_{L^2({\mathbb R}^q,\tilde{E})} = \langle A^*Au,u \rangle_{L^2({\mathbb R}^q,
E)} \leq \|A^*A\| \|u\|^2_{L^2({\mathbb R}^q,E)},
$$
where we used the induction hypothesis according to which $A^*A$ is $L^2$-continuous. This completes the inductive argument.

Finally, let $A = \Op_0(a)$ with $a(y,\eta) \in S^{(0,0)}_{\varrho,\delta}({\mathbb R}^q\times{\mathbb R}^q;E,\tilde{E})$, and choose
$$
M > 2\sup_{(y,\eta) \in {\mathbb R}^q\times{\mathbb R}^q}\|a(y,\eta)^*a(y,\eta)\|_{{\mathscr L}(E)}.
$$
Then
$$
c(y,\eta) = (M - a(y,\eta)^*a(y,\eta))^{1/2} \in S^{(0,0)}_{\varrho,\delta}({\mathbb R}^q\times{\mathbb R}^q;E,E).
$$
To see this observe that $M - a(y,\eta)^*a(y,\eta) \in {\mathscr L}(E)$ is selfadjoint with spectrum contained in $[M/2,M]$, so we can write
$$
c(y,\eta) = \frac{1}{2\pi i} \int_{\Gamma} \lambda^{1/2}(\lambda - M + a(y,\eta)^*a(y,\eta))^{-1}\,d\lambda
$$
with a fixed contour $\Gamma$ that is contained in the right half-plane ${\mathbb R}e(\lambda) > 0$, encloses $[M/2,M]$, and is independent of $(y,\eta) \in {\mathbb R}^q\times{\mathbb R}^q$. We have
$$
M - a(y,\eta)^*a(y,\eta) \in S^{(0,0)}_{\varrho,\delta}({\mathbb R}^q\times{\mathbb R}^q;E,E),
$$
and by the spectral theorem we have
$$
\sup\{\|(\lambda - M + a(y,\eta)^*a(y,\eta))^{-1}\|_{{\mathscr L}(E)} : (y,\eta) \in {\mathbb R}^q\times{\mathbb R}^q,\; \lambda \in \Gamma\} < \infty.
$$
The Dunford integral representation for $c(y,\eta)$ and differentiation under the integral sign therefore show that $c(y,\eta) \in S^{(0,0)}_{\varrho,\delta}({\mathbb R}^q\times{\mathbb R}^q;E,E)$ as was claimed. From the symbol and operator calculus (Theorems~\ref{TKCompositionTheorem} and \ref{TKAdjointTheorem}) we now obtain
$$
\Op_0(c)^*\circ \Op_0(c) = M - \Op_0(a)^*\circ \Op_0(a) + R
$$
for some $R \in \Psi^{(\mu,0)}_{\varrho,\delta}({\mathbb R}^q;E,E)$ with $\mu < 0$. In particular,
$$
\|\Op_0(a)u\|^2_{L^2({\mathbb R}^q,\tilde{E})} \leq M\|u\|_{L^2({\mathbb R}^q,E)}^2 + \langle Ru,u \rangle_{L^2({\mathbb R}^q,E)},
$$
and by the $L^2$-boundedness of $R$ proved earlier we conclude that $A = \Op_0(a)$ is $L^2$-bounded. Lastly, the continuity of the map
$$
\Psi^{(0,0)}_{\varrho,\delta}({\mathbb R}^q;E,\tilde{E}) \to {\mathscr L}(L^2({\mathbb R}^q,E),L^2({\mathbb R}^q,\tilde{E}))
$$
follows readily from the closed graph theorem, which finishes the proof.
\end{proof}

To address compactness we need two lemmas. Let ${\mathscr K}$ denote the operator ideal of compact operators in the sequel.

\begin{lemma}\label{TKcontinuouskernelscompact}
Let $k(y,y') \in C_c({\mathbb R}^q\times{\mathbb R}^q,{\mathscr L}(E,\tilde{E}))$, and suppose that $k(y,y') \in {\mathscr K}(E,\tilde{E})$ for all $y,y' \in {\mathbb R}^q\times{\mathbb R}^q$. Then the integral operator
$$
Au(y) = \int_{{\mathbb R}^q}k(y,y')u(y')\,dy'
$$
belongs to ${\mathscr K}(L^2({\mathbb R}^q,E),L^2({\mathbb R}^q,\tilde{E}))$.
\end{lemma}
\begin{proof}
We recall the standard proof (see also \cite[Proposition 2.1]{TKLuke}).

Choose $R > 0$ such that $\supp(k) \subset (-R,R)^{q} \times (-R,R)^{q}$, and consider the operator
$$
A\big|_{[-R,R]^q} : L^2([-R,R]^q,E) \to L^2([-R,R]^q,\tilde{E}).
$$
Clearly, $A = \textup{Ext} \circ A\big|_{[-R,R]^q} \circ r$, where $r : L^2({\mathbb R}^q,E) \to L^2([-R,R]^q,E)$ is the (continuous) restriction operator, and $\textup{Ext} : L^2([-R,R]^q,\tilde{E}) \to L^2({\mathbb R}^q,\tilde{E})$ is the (continuous) trivial extension operator. Hence the claim reduces to proving that any integral operator $B$ with continuous kernel $k_B(y,y') \in C([-R,R]^q\times[-R,R]^{q},{\mathscr K}(E,\tilde{E}))$ is compact in ${\mathscr L}(L^2([-R,R]^q,E),L^2([-R,R]^q,\tilde{E}))$.

In view of
$$
\|B\|_{{\mathscr L}(L^2,L^2)} \leq (2R)^{q}\sup\{\|k_B(y,y')\| : y,y' \in [-R,R]^q\}
$$
the map
$$
C([-R,R]^q\times[-R,R]^{q},{\mathscr K}(E,\tilde{E})) \ni k_B \mapsto B \in {\mathscr L}(L^2,L^2)
$$
is continuous. Using compactness, uniform continuity of the kernel, and partitions of unity shows that every compact kernel $k_B$ is the uniform limit of kernels in the algebraic tensor product $C([-R,R]^q)\otimes {\mathscr K}(E,\tilde{E}) \otimes C([-R,R]^q)$. If $k_B(y,y') = \phi(y) K \psi(y')$ is a pure tensor, then
\begin{equation}\label{TKBauxcompact}
B = M_{\phi} \circ K \circ Q_{\psi} : L^2([-R,R]^q,E) \to L^2([-R,R]^q,\tilde{E}),
\end{equation}
where
\begin{gather*}
Q_{\psi} : L^2([-R,R]^q,E) \to E, \; Q_{\psi}u = \int \psi(y')u(y')\,dy', \\
M_\phi : \tilde{E} \to L^2([-R,R]^q,\tilde{E}), \; [M_{\phi}\tilde{e}](y) = \phi(y)\tilde{e},
\end{gather*}
are continuous, and $K : E \to \tilde{E}$ is compact. Thus \eqref{TKBauxcompact} is compact, which shows that $k_B \mapsto B$ maps
$$
C([-R,R]^q)\otimes {\mathscr K}(E,\tilde{E}) \otimes C([-R,R]^q) \to {\mathscr K}(L^2,L^2).
$$
Consequently, by density and continuity, all integral operators with continuous compact kernels are compact.
\end{proof}

\begin{lemma}\label{TKstandardpseudoscompact}
Every $a(y,\eta) \in S_{\varrho,\delta}^{(0,0)}({\mathbb R}^q\times{\mathbb R}^q;E,\tilde{E})_{{\mathscr K},(0)}$ induces a compact operator
$$
A = \Op_0(a) : L^2({\mathbb R}^q,E) \to L^2({\mathbb R}^q,\tilde{E}).
$$
\end{lemma}
\begin{proof}
(See \cite[Proposition 2.1]{TKLuke}). Because
$$
S_{\varrho,\delta}^{(0,0)}({\mathbb R}^q\times{\mathbb R}^q;E,\tilde{E}) \ni a \mapsto \Op_0(a) \in {\mathscr L}(L^2({\mathbb R}^q,E),L^2({\mathbb R}^q,\tilde{E}))
$$
is continuous it suffices to show that $\Op_0(a) \in  {\mathscr K}(L^2({\mathbb R}^q,E),L^2({\mathbb R}^q,\tilde{E}))$ for any $a \in C_c^{\infty}({\mathbb R}^q\times{\mathbb R}^q,{\mathscr K}(E,\tilde{E}))$. Because of the ideal property of compact operators and Plancherel this is equivalent to the compactness of the operator
$$
B = \Op_0(a)\circ{\mathscr F}^{-1} : L^2({\mathbb R}^q,E) \to L^2({\mathbb R}^q,\tilde{E}).
$$
We have
$$
Bu(y) = \int_{{\mathbb R}^q}k_B(y,y')u(y')\,dy'
$$
with $k_B(y,y') = (2\pi)^{-q} e^{iyy'}a(y,y') \in C_c^{\infty}({\mathbb R}^q\times{\mathbb R}^q,{\mathscr K}(E,\tilde{E}))$. Thus $B$ is compact by Lemma~\ref{TKcontinuouskernelscompact}.
\end{proof}

\begin{theorem}
Every $A \in \Psi^{\vec{\mu}}_{\varrho,\delta}({\mathbb R}^q;(E,\kappa),(\tilde{E},\tilde{\kappa}))_{{\mathscr K}}$ induces a compact operator
$$
A : {\mathcal W}^{\vec{s}}({\mathbb R}^q,E) \to {\mathcal W}^{\vec{s'}}({\mathbb R}^q,\tilde{E})
$$
for all $\vec{s},\vec{s'} \in {\mathbb R}^2$ such that $\vec{\mu} < \vec{s} - \vec{s'}$, where this inequality is to hold componentwise.
\end{theorem}
\begin{proof}
Arguing as in the proof of Theorem~\ref{TKBoundednessTheorem} reduces the claim to the case that $\vec{s}=\vec{s'}=(0,0)$, $\vec{\mu} < (0,0)$, and trivial group actions $\kappa_{\varrho} \equiv \textup{Id}_E$ and $\tilde{\kappa}_{\varrho} \equiv \textup{Id}_{\tilde{E}}$ (eliminating the group actions is based on Proposition~\ref{TKBddnessWsScale}). But this case is treated in Lemma~\ref{TKstandardpseudoscompact}, taking into account Lemma~\ref{TKdensitylemma}.
\end{proof}


\section{Trace class operators}\label{TKsec-TraceClass}

We remind the reader that ${\mathscr C}_1$ denotes the Banach operator ideal of Schatten--von Neumann operators with $\ell^1$-summable approximation numbers (i.e. the trace class operators in case of operators acting on a single Hilbert space). Moreover, we always assume that $0 \leq \delta < \varrho \leq 1$.

The proof of trace class properties presented here in Lemmas~\ref{TKintegralC1}--\ref{TKwithcutoffs} and Proposition~\ref{TKTraceClassOperators} follows Widom \cite{TKWidom90}.

\begin{lemma}\label{TKintegralC1}
Let $H$ and $\tilde{H}$ be separable complex Hilbert spaces, and let $\sigma : {\mathbb R}^q \to {\mathscr L}(H,\tilde{H})$ be a weakly measurable operator function. Suppose that $\sigma(\eta) \in {\mathscr C}_1(H,\tilde{H})$ for almost all $\eta \in {\mathbb R}^q$, and suppose further that there exists a function $g \in L^1({\mathbb R}^q)$ such that $\|\sigma(\eta)\|_{{\mathscr C}_1} \leq g(\eta)$ for almost all $\eta \in {\mathbb R}^q$. Define the operator $A_{\sigma} : H \to \tilde{H}$ via
$$
\langle A_{\sigma}h,\tilde{h} \rangle_{\tilde{H}} = \int_{{\mathbb R}^q} \langle \sigma(\eta)h,\tilde{h} \rangle_{\tilde{H}}\,d\eta
$$
for $h \in H$ and $\tilde{h} \in \tilde{H}$. Then $A_{\sigma} \in {\mathscr C}_1(H,\tilde{H})$ with $\|A\|_{{\mathscr C}_1} \leq \|g\|_{L^1}$.
\end{lemma}
\begin{proof}
Because $\|\sigma(\eta)\|_{{\mathscr L}(H,\tilde{H})} \leq \|\sigma(\eta)\|_{{\mathscr C}_1(H,\tilde{H})}$ we first obtain that
$$
\Bigl|\int_{{\mathbb R}^q} \langle \sigma(\eta)h,\tilde{h} \rangle_{\tilde{H}}\,d\eta\Bigr| \leq \|g\|_{L^1}\|h\|_H\|\tilde{h}\|_{\tilde{H}},
$$
which shows that $A_{\sigma} \in {\mathscr L}(H,\tilde{H})$ is well-defined.

Now let $\{h_j\} \subset H$ and $\{\tilde{h}_j\} \subset \tilde{H}$ be finite sets of orthonormal vectors. Then
$$
\sum |\langle A_{\sigma}h_j,\tilde{h}_j \rangle_{\tilde{H}}| \leq \int_{{\mathbb R}^q} \underbrace{\sum|\langle \sigma(\eta)h_j,\tilde{h}_j \rangle_{\tilde{H}}|}_{\leq \|\sigma(\eta)\|_{{\mathscr C}_1}}\,d\eta \leq \|g\|_{L^1}.
$$
Passing to the supremum over all such sets of orthonormal vectors we obtain that $A_{\sigma} \in {\mathscr C}_1(H,\tilde{H})$ with $\|A\|_{{\mathscr C}_1} \leq \|g\|_{L^1}$ as asserted.
\end{proof}

\begin{lemma}\label{TKnoetadependence}
Let $\varphi,\psi \in L^2({\mathbb R}^q)$, and let both $a,{\mathscr F} a \in L^1({\mathbb R}^q,{\mathscr C}_1(E,\tilde{E}))$. Consider the map $G : L^2({\mathbb R}^q,E) \to L^2({\mathbb R}^q,\tilde{E})$,
$G u(y) = \varphi(y)a(y)\int_{{\mathbb R}^q}\psi(y')u(y')\,dy'$. Then $G \in {\mathscr C}_1(L^2({\mathbb R}^q,E),L^2({\mathbb R}^q,\tilde{E}))$, and
$$
\|G\|_{{\mathscr C}_1(L^2,L^2)} \leq (2\pi)^{-q}\|\varphi\|_{L^2}\|\psi\|_{L^2} \|{\mathscr F} a\|_{L^1({\mathbb R}^q,{\mathscr C}_1(E,\tilde{E}))}.
$$
\end{lemma}
\begin{proof}
By assumption on $a$, both $a(y)$ and $\hat{a}(\eta) = {\mathscr F} a(\eta)$ are continuous and bounded functions on ${\mathbb R}^q$, and we have $a(y) = (2\pi)^{-q}\int_{{\mathbb R}^q}e^{iy\eta}\hat{a}(\eta)\,d\eta$. Write
$$
Gu(y) = (2\pi)^{-q}\int_{{\mathbb R}^q}e^{iy\eta}\varphi(y)\hat{a}(\eta)\,d\eta \int_{{\mathbb R}^q}\psi(y')u(y')\,dy'.
$$
For $\eta \in {\mathbb R}^q$ consider the operator
\begin{gather*}
\sigma(\eta) : E \to L^2({\mathbb R}^q,\tilde{E}), \\
\sigma(\eta)e = [y \mapsto (2\pi)^{-q}e^{iy\eta}\varphi(y)\hat{a}(\eta)e].
\end{gather*}
By Lebesgue's dominated convergence theorem, the function $\eta \mapsto \sigma(\eta)e$ is continuous with values in $L^2({\mathbb R}^q,\tilde{E})$ for each $e \in E$. In particular, the operator function $\sigma : {\mathbb R}^q \to {\mathscr L}(E,L^2({\mathbb R}^q,\tilde{E}))$ is weakly measurable. Note, furthermore, that $\sigma(\eta)$ is itself the composition of the operators $\tilde{E} \ni \tilde{e} \mapsto (2\pi)^{-q}e^{iy\eta}\varphi(y)\tilde{e} \in L^2({\mathbb R}^q,\tilde{E})$ and $\hat{a}(\eta) : E \to \tilde{E}$. The latter belongs to ${\mathscr C}_1(E,\tilde{E})$ by assumption, and the former is bounded with operator norm at most $(2\pi)^{-q}\|\varphi\|_{L^2}$. Thus $\sigma(\eta) \in {\mathscr C}_1(E,L^2({\mathbb R}^q,\tilde{E}))$ for all $\eta \in {\mathbb R}^q$ with $\|\sigma(\eta)\|_{{\mathscr C}_1} \leq (2\pi)^{-q}\|\hat{a}(\eta)\|_{{\mathscr C}_1}\|\varphi\|_{L^2}$. By Lemma~\ref{TKintegralC1}, the operator
$$
E \ni e \mapsto [y \mapsto (2\pi)^{-q}\int_{{\mathbb R}^q}e^{iy\eta}\varphi(y)\hat{a}(\eta)e\,d\eta] \in L^2({\mathbb R}^q,\tilde{E})
$$
belongs to ${\mathscr C}_1(E,L^2({\mathbb R}^q,\tilde{E}))$ with ${\mathscr C}_1$-norm at most $(2\pi)^{-q}\|\varphi\|_{L^2}\|\hat{a}\|_{L^1({\mathbb R}^q,{\mathscr C}_1(E,\tilde{E}))}$.
Finally, the operator $G$ is the composition of this map and the operator $L^2({\mathbb R}^q,E) \ni u \mapsto \int_{{\mathbb R}^q}\psi(y')u(y')\,dy' \in E$, which is continuous with operator norm at most $\|\psi\|_{L^2}$. The lemma is proved.
\end{proof}

\begin{lemma}\label{TKwithcutoffs}
Let $a(y,\eta) \in {\mathscr S}({\mathbb R}^q\times{\mathbb R}^q,{\mathscr C}_1(E,\tilde{E}))$, $p = \lceil (q+1)/4 \rceil$, and $\mu' < -q$. Consider the operator
$$
Au(y) = (2\pi)^{-q}\int_{{\mathbb R}^q}e^{iy\eta}\langle y \rangle^{-2p}a(y,\eta)\int_{{\mathbb R}^q}e^{-iy'\eta}\langle y' \rangle^{-2p}u(y')\,dy'\,d\eta
$$
for $u \in {\mathscr S}({\mathbb R}^q,E)$. Then $A \in {\mathscr C}_1(L^2({\mathbb R}^q,E),L^2({\mathbb R}^q,\tilde{E}))$ with
\begin{align*}
\|A\|_{{\mathscr C}_1(L^2,L^2)} &\leq (2\pi)^{-2q}\Bigl(\int_{{\mathbb R}^q}\langle y' \rangle^{-4p}\,dy'\Bigr)\|{\mathscr F}_{y\to\eta'}a(\eta',\eta)\|_{L^1({\mathbb R}^q_{\eta'}\times{\mathbb R}^q_{\eta},{\mathscr C}_1(E,\tilde{E}))} \\
&\leq C(\mu',q) \sup_{(y,\eta) \in {\mathbb R}^{2q}}\langle y \rangle^{4p}\langle \eta \rangle^{-\mu'}\|(1-\Delta_y)^{2p}a(y,\eta)\|_{{\mathscr C}_1(E,\tilde{E})},
\end{align*}
where
$$
C(\mu',q) = (2\pi)^{-2q}\Bigl(\int_{{\mathbb R}^q}\langle y' \rangle^{-4p}\,dy'\Bigr)^3\Bigl(\int_{{\mathbb R}^q}\langle \eta \rangle^{\mu'}\,d\eta\Bigr).
$$
\end{lemma}
\begin{proof}
Fix $\eta \in {\mathbb R}^q$ and consider the map
\begin{gather*}
\sigma(\eta) : L^2({\mathbb R}^q,E) \to L^2({\mathbb R}^q,\tilde{E}), \\
\sigma(\eta)u(y) = (2\pi)^{-q}e^{iy\eta}\langle y \rangle^{-2p}a(y,\eta)\int_{{\mathbb R}^q}e^{-iy'\eta}\langle y' \rangle^{-2p}u(y')\,dy'.
\end{gather*}
Lemma~\ref{TKnoetadependence} is applicable to $\sigma(\eta)$. We obtain that $\sigma(\eta) \in {\mathscr C}_1(L^2({\mathbb R}^q,E),L^2({\mathbb R}^q,\tilde{E}))$ with
$$
\|\sigma(\eta)\|_{{\mathscr C}_1(L^2,L^2)} \leq (2\pi)^{-2q}\Bigl(\int_{{\mathbb R}^q}\langle y' \rangle^{-4p}\,dy'\Bigr)\|{\mathscr F}_{y\to\eta'}a(\eta',\eta)\|_{L^1({\mathbb R}^q_{\eta'},{\mathscr C}_1(E,\tilde{E}))}.
$$
An application of Lebesgue's dominated convergence theorem shows that the function $\eta \mapsto \sigma(\eta)u$ depends continuously on $\eta \in {\mathbb R}^q$ taking values in $L^2({\mathbb R}^q,\tilde{E})$ for each $u \in L^2({\mathbb R}^q,E)$. In particular, $\sigma : {\mathbb R}^q \to {\mathscr L}(L^2({\mathbb R}^q,E),L^2({\mathbb R}^q,\tilde{E}))$ is a weakly measurable operator function.
By Lemma~\ref{TKintegralC1} the operator $A = A_{\sigma}$ defined by $\sigma$ belongs to ${\mathscr C}_1(L^2({\mathbb R}^q,E),L^2({\mathbb R}^q,\tilde{E}))$ with
$$
\|A\|_{{\mathscr C}_1(L^2,L^2)} \leq (2\pi)^{-2q}\Bigl(\int_{{\mathbb R}^q}\langle y' \rangle^{-4p}\,dy'\Bigr)\|{\mathscr F}_{y\to\eta'}a(\eta',\eta)\|_{L^1({\mathbb R}^q_{\eta'}\times{\mathbb R}^q_{\eta},{\mathscr C}_1(E,\tilde{E}))},
$$
proving the claim.
\end{proof}

\begin{proposition}\label{TKTraceClassOperators}
Let $p = \lceil (q+1)/4 \rceil$, and let $\mu_1 < -q-4p\delta$ and $\mu_2 < -8p$. For every $a(y,\eta) \in S_{\varrho,\delta}^{(\mu_1,\mu_2)}({\mathbb R}^q\times{\mathbb R}^q;E,\tilde{E})_{{\mathscr C}_1}$ the operator
$$
Au(y) = \Op_0(a)u(y) = (2\pi)^{-q}\int_{{\mathbb R}^q} e^{iy\eta}a(y,\eta)\hat{u}(\eta)\,d\eta, \quad u \in {\mathscr S}({\mathbb R}^q,E),
$$
belongs to ${\mathscr C}_1(L^2({\mathbb R}^q,E),L^2({\mathbb R}^q,\tilde{E}))$. The map
$$
S_{\varrho,\delta}^{(\mu_1,\mu_2)}({\mathbb R}^q\times{\mathbb R}^q;E,\tilde{E})_{{\mathscr C}_1} \ni a \mapsto \Op_0(a) \in {\mathscr C}_1(L^2({\mathbb R}^q,E),L^2({\mathbb R}^q,\tilde{E}))
$$
is continuous. Moreover, if $E = \tilde{E}$, then
$$
\Tr_{L^2({\mathbb R}^q,E)}[\Op_0(a)] = \frac{1}{(2\pi)^q}\iint_{{\mathbb R}^{2q}} \Tr_{E}(a(y,\eta))\,dyd\eta.
$$
\end{proposition}
\begin{proof}
Consider first $a(y,\eta) \in {\mathscr S}({\mathbb R}^q\times{\mathbb R}^q,{\mathscr C}_1(E,\tilde{E}))$. We have
\begin{gather*}
Au(y) = (2\pi)^{-q}\int_{{\mathbb R}^q} e^{i(y\eta}a(y,\eta)\int_{{\mathbb R}^q}e^{-iy'\eta}u(y')\,dy'\,d\eta \\
= (2\pi)^{-q}\int_{{\mathbb R}^q} e^{iy\eta}a(y,\eta)(1-\Delta_{\eta})^p\int_{{\mathbb R}^q}e^{-iy'\eta}\langle y' \rangle^{-2p}u(y')\,dy'\,d\eta \\
= (2\pi)^{-q}\int_{{\mathbb R}^q} e^{iy\eta} \langle y \rangle^{-2p}\underbrace{[\langle y \rangle^{2p}({\mathcal L}^pa)(y,\eta)]}_{= b(y,\eta)} \int_{{\mathbb R}^q}e^{-iy'\eta}\langle y' \rangle^{-2p}u(y')\,dy'\,d\eta,
\end{gather*}
where
$$
{\mathcal L} = e^{-iy\eta}(1-\Delta_{\eta})e^{iy\eta} = 1 + \sum_{j=1}^q(y_j + D_{\eta_j})^2.
$$
By Lemma~\ref{TKwithcutoffs}, $A \in {\mathscr C}_1(L^2({\mathbb R}^q,E),L^2({\mathbb R}^q,\tilde{E}))$, and for any $\mu' < -q$ the ${\mathscr C}_1$-norm of $A$ can be estimated by a continuous seminorm of $(1-\Delta_y)^{2p}b(y,\eta)$ in $S^{(\mu',-4p)}_{\varrho,\delta}({\mathbb R}^q\times{\mathbb R}^q;E,\tilde{E})_{{\mathscr C}_1}$. Because the map $a(y,\eta) \mapsto (1-\Delta_y)^{2p}b(y,\eta)$ is continuous in
$$
S_{\varrho,\delta}^{\mu'-4p\delta,-8p}({\mathbb R}^q\times{\mathbb R}^q;E,\tilde{E})_{{\mathscr C}_1,(0)} \to S_{\varrho,\delta}^{\mu',-4p}({\mathbb R}^q\times{\mathbb R}^q;E,\tilde{E})_{{\mathscr C}_1,(0)},
$$
we see that the ${\mathscr C}_1$-norm of $A$ can be estimated by a continuous seminorm of $a(y,\eta)$ in $S_{\varrho,\delta}^{(\mu'-4p\delta,-8p)}({\mathbb R}^q\times{\mathbb R}^q;E,\tilde{E})_{{\mathscr C}_1}$. Therefore, the map
$$
S_{\varrho,\delta}^{(\mu'-4p\delta,-8p)}({\mathbb R}^q\times{\mathbb R}^q;E,\tilde{E})_{{\mathscr C}_1,(0)} \ni a \mapsto \Op_0(a) \in  {\mathscr C}_1(L^2({\mathbb R}^q,E),L^2({\mathbb R}^q,\tilde{E}))
$$
is well-defined and continuous. Because $\mu' < -q$ is arbitrary we obtain the first part of the proposition by applying Lemma~\ref{TKdensitylemma}.

It remains to show the asserted trace formula in the case that $E = \tilde{E}$. Since both functionals
\begin{align*}
S_{\varrho,\delta}^{(\mu_1,\mu_2)}({\mathbb R}^q\times{\mathbb R}^q;E,E)_{{\mathscr C}_1,(0)} \ni a &\mapsto [\Tr_{L^2({\mathbb R}^q,E)}\circ\Op_0](a) \in {\mathbb C} \\
S_{\varrho,\delta}^{(\mu_1,\mu_2)}({\mathbb R}^q\times{\mathbb R}^q;E,E)_{{\mathscr C}_1,(0)} \ni a &\mapsto \frac{1}{(2\pi)^q}\iint_{{\mathbb R}^{2q}} \Tr_{E}(a(y,\eta))\,dyd\eta \in {\mathbb C}
\end{align*}
are continuous, it suffices to show that they are equal on the dense subset
$$
[C_c^{\infty}({\mathbb R}_y^q)\otimes {\mathscr F}_{y'\to\eta}(C_c^{\infty}({\mathbb R}_{y'}^q))] \otimes {\mathscr C}_1(E),
$$
but this is evident in view of the multiplicativity of the trace functional for tensor products of trace class operators. By Lemma~\ref{TKdensitylemma} we can drop the subscript $(0)$ from the symbol spaces.
\end{proof}

\begin{theorem}\label{TKTraceClassTheorem}
Let $p = \lceil (q+1)/4 \rceil$. Every $A \in \Psi^{\vec{\mu}}_{\varrho,\delta}({\mathbb R}^q;(E,\kappa),(\tilde{E},\tilde{\kappa}))_{{\mathscr C}_1}$ induces a ${\mathscr C}_1$-operator
$$
A : {\mathcal W}^{\vec{s}}({\mathbb R}^q,E) \to {\mathcal W}^{\vec{s'}}({\mathbb R}^q,\tilde{E})
$$
for all $\vec{s},\vec{s'} \in {\mathbb R}^2$ such that $\vec{\mu} + (q+4p\delta,8p) < \vec{s} - \vec{s'}$, where this inequality is to hold componentwise.
\end{theorem}
\begin{proof}
Arguing as in the proof of Theorem~\ref{TKBoundednessTheorem} reduces the claim to the case that $\vec{s}=\vec{s'}=(0,0)$ and trivial group actions $\kappa_{\varrho} \equiv \textup{Id}_E$ and $\tilde{\kappa}_{\varrho} \equiv \textup{Id}_{\tilde{E}}$ (note that eliminating the group actions is based on Proposition~\ref{TKBddnessWsScale}). But this case is discussed in Proposition~\ref{TKTraceClassOperators}.
\end{proof}

\begin{remark}
Theorem~\ref{TKTraceClassTheorem} is certainly not optimal, especially regarding the growth condition on the variable $y \in {\mathbb R}^q$ as $|y| \to \infty$. However, in the case that $\delta = 0$ the condition on the order $\mu_1$ is the best possible as the scalar case shows. In applications concerning compact manifolds with edge singularities, the abstract edge calculus is used to describe the local structure of operators near the edge. In this case, one may generally assume without loss of generality that the symbol $a(y,\eta)$ of an operator $A = \Op_0(a)$ has compact support in $y$, so the growth condition of Theorem~\ref{TKTraceClassTheorem} with respect to $y \in {\mathbb R}^q$ is satisfied in such applications.
\end{remark}


\begin{appendix}

\section{Banach operator ideals in Hilbert spaces}

We refer to \cite{TKPietschIdeals,TKPietsch,TKSimon} for background on the theory of operator ideals. We remind the reader that all Hilbert spaces are assumed to be complex and separable.

\begin{definition}
Suppose that for every pair of Hilbert spaces $E$ and $\tilde{E}$ one is given a subset ${\mathscr I}(E,\tilde{E}) \subset {\mathscr L}(E,\tilde{E})$ with the following properties:
\begin{enumerate}
\item All finite-rank operators $F : E \to \tilde{E}$ belong to ${\mathscr I}(E,\tilde{E})$.
\item Whenever $A,B \in {\mathscr I}(E,\tilde{E})$, then $A + B \in {\mathscr I}(E,\tilde{E})$.
\item We have $GAH \in {\mathscr I}(E_0,E_1)$ whenever $G \in {\mathscr L}(\tilde{E},E_1)$, $A \in {\mathscr I}(E,\tilde{E})$, and $H \in {\mathscr L}(E_0,E)$.
\end{enumerate}
Then the collection
$$
{\mathscr I} = \bigcup_{E,\tilde{E}}{\mathscr I}(E,\tilde{E})
$$
is called an operator ideal (in the category of Hilbert spaces).

An operator ideal ${\mathscr I}$ is called normed if there is a function $\|\cdot\|_{{\mathscr I}} : {\mathscr I} \to {\mathbb R}$ that restricts to a norm $\|\cdot\|_{{\mathscr I}(E,\tilde{E})}$ on ${\mathscr I}(E,\tilde{E})$ for all Hilbert spaces $E$ and $\tilde{E}$ such that the following holds:
\begin{enumerate}
\item $\| e'\otimes \tilde{e} \|_{{\mathscr I}(E,\tilde{E})} = \|e'\|_{E'}\|\tilde{e}\|_{\tilde{E}}$ for all $e' \in E'$ and all $\tilde{e} \in \tilde{E}$.
\item $\|GAH\|_{{\mathscr I}(E_0,E_1)} \leq \|G\|_{{\mathscr L}(\tilde{E},E_1)} \|A\|_{{\mathscr I}(E,\tilde{E})} \|H\|_{{\mathscr L}(E_0,E)}$ for all operators $G \in {\mathscr L}(\tilde{E},E_1)$, $A \in {\mathscr I}(E,\tilde{E})$, and $H \in {\mathscr L}(E_0,E)$.
\end{enumerate}

A normed operator ideal ${\mathscr I}$ is a Banach operator ideal if ${\mathscr I}(E,\tilde{E})$ is complete with the norm $\|\cdot\|_{{\mathscr I}(E,\tilde{E})}$ for all Hilbert spaces $E$ and $\tilde{E}$.
\end{definition}

We remark that it is a consequence of the axioms that
$$
\|A\|_{{\mathscr L}(E,\tilde{E})} \leq \|A\|_{{\mathscr I}(E,\tilde{E})}
$$
for all $A \in {\mathscr I}(E,\tilde{E})$, see \cite[Proposition~6.1.4]{TKPietschIdeals}. In particular, the embedding
$$
({\mathscr I}(E,\tilde{E}),\|\cdot\|_{{\mathscr I}}) \hookrightarrow ({\mathscr L}(E,\tilde{E}),\|\cdot\|_{{\mathscr L}(E,\tilde{E})})
$$
is continuous.

\medskip

There are two trivial Banach operator ideals, ${\mathscr I} = {\mathscr L}$ and ${\mathscr I } = {\mathscr K}$: The first consists of ${\mathscr I}(E,\tilde{E}) = {\mathscr L}(E,\tilde{E})$ (all bounded operators), the second one has ${\mathscr I}(E,\tilde{E}) = {\mathscr K}(E,\tilde{E})$ (all compact operators) for Hilbert spaces $E$ and $\tilde{E}$. Both ${\mathscr L}$ and ${\mathscr K}$ are Banach operator ideals with respect to the usual operator norm.

In applications of spectral and index theory, the Schatten-von Neumann classes ${\mathscr C}_p$, $1 \leq p < \infty$, occur frequently. These are Banach operator ideals with respect to the Schatten $p$-norms. For $p = 1$ and $E = \tilde{E}$, the class ${\mathscr C}_1(E,E)$ coincides with the trace class operators, and for $p = 2$ the class ${\mathscr C}_2$ specializes to the Hilbert-Schmidt operators.

Let $A \in {\mathscr L}(E,\tilde{E})$. For $r \in {\mathbb N}$ the $r$-th approximation number of $A$ is defined as
$$
\alpha_r(A) = \inf\{\|A - F\|_{{\mathscr L}(E,\tilde{E})} : F \in {\mathscr L}(E,\tilde{E}),\; \dim R(F) < r\}.
$$

\begin{definition}
For $1 \leq p < \infty$ and Hilbert spaces $E$ and $\tilde{E}$ define the Schatten-von Neumann class ${\mathscr C}_p(E,\tilde{E})$ as the space of all operators $A \in {\mathscr L}(E,\tilde{E})$ such that
$$
\|A\|_{{\mathscr C}_p} = \Bigl(\sum_{r=1}^{\infty}\alpha_r(A)^p\Bigr)^{1/p} < \infty.
$$
\end{definition}

The following theorem summarizes some of the properties of this class.

\begin{theorem}\label{TKCpThm}
\begin{enumerate}
\item An operator $A : E \to \tilde{E}$ belongs to the class ${\mathscr C}_p(E,\tilde{E})$, $1 \leq p < \infty$, if and only if
$$
\sup\Bigl\{ \Bigl(\sum |\langle Ae_k,\tilde{e}_k \rangle|^p\Bigr)^{1/p} : (e_k) \subset E,\; (\tilde{e}_k) \subset \tilde{E} \textup{ are finite orthonormal}\Bigr\} < \infty.
$$
In this case,
$$
\|A\|_{{\mathscr C}_p} = \sup\Bigl\{ \Bigl(\sum |\langle Ae_k,\tilde{e}_k \rangle|^p\Bigr)^{1/p} : (e_k) \subset E,\; (\tilde{e}_k) \subset \tilde{E} \textup{ are finite orthonormal}\Bigr\}.
$$
\item ${\mathscr C}_p$ is a Banach operator ideal with norm $\|\cdot\|_{{\mathscr C}_p}$.
\item Composition of operators induces a continuous map
$$
{\mathscr C}_p(E_1,E_2) \times {\mathscr C}_{p'}(E_0,E_1) \to {\mathscr C}_1(E_0,E_2),
$$
where $1 < p,p' < \infty$ with $\frac{1}{p} + \frac{1}{p'} = 1$. More precisely, we have
$$
\|AB\|_{{\mathscr C}_1} \leq \|A\|_{{\mathscr C}_p}\|B\|_{{\mathscr C}_{p'}}.
$$
\end{enumerate}
\end{theorem}

Theorem~\ref{TKCpThm} implies that the classes ${\mathscr C}_p$ are invariant under taking Hilbert space adjoints. More precisely, we have $\|A^*\|_{{\mathscr C}_p(E,\tilde{E})} = \|A\|_{{\mathscr C}_p(\tilde{E},E)}$.

\end{appendix}



\end{document}